\documentclass[]{amsart}
\usepackage[latin2]{inputenc}
\usepackage{t1enc}
\usepackage[mathscr]{eucal}
\usepackage{graphicx}
\usepackage{graphics}
\numberwithin{equation}{section}
\usepackage{hyperref}

\hypersetup{
 colorlinks=true,       
    linkcolor=cyan,          
    citecolor=green,        
    filecolor=magenta,      
   urlcolor=cyan           
   }

\allowdisplaybreaks
\newcommand{\eq}{\begin{equation}}
\newcommand{\en}{\end{equation}}

\DeclareSymbolFont{bbold}{U}{bbold}{m}{n}
\DeclareSymbolFontAlphabet{\mathbbold}{bbold}
\newcommand{\ind}{\mathbbold{1}}

\makeatletter

\newcommand{\Rmnum}[1]{\expandafter\@slowromancap\romannumeral #1@}
\makeatother

\newtheorem{thm}{Theorem}[section]
\newtheorem{prop}[thm]{Proposition}
\newtheorem{lem}[thm]{Lemma}
\newtheorem{cor}[thm]{Corollary}
\theoremstyle{definition}
\newtheorem{remark}[thm]{Remark}
\newtheorem{defn}[thm]{Definition}

\newtheorem{notation}[thm]{Notation}

\numberwithin{equation}{section}
\numberwithin{figure}{section}

\newcommand{\reals}{\mathbb{R}}
\newcommand{\ints}{\mathbb{Z}}
\newcommand{\nats}{\mathbb{N}}

\newcommand{\lra}{\leftrightarrow}
\newcommand{\ra}{\rightarrow}
\renewcommand{\P}{\mathbb{P}}
\newcommand{\Q}{\mathbb{Q}}
\newcommand{\E}{\mathbb{E}}
\newcommand{\ed}{\,{\buildrel d \over =}\,}

\renewcommand{\and}{ \quad \text{and} \quad }

\begin{document}

\title[Lipschitz minorants]{Excursions away from the Lipschitz minorant of a L\'evy process}


\author{Steven N. Evans}
\address{Department of Statistics \#3860 \\ 367 Evans Hall \\ University of California at Berkeley \\ Berkeley, CA 94720-3860 \\ USA}
\email{evans@stat.berkeley.edu}
\thanks{SNE supported in part by NSF grant DMS-1512933
and NIH grant 1R01GM109454-01.}

\author{Mehdi Ouaki}
\address{Department of Statistics \#3860 \\ 367 Evans Hall \\ University of California at Berkeley \\ Berkeley, CA 94720-3860 \\ USA}
\email{mouaki@berkeley.edu}
\thanks{}



\date{\today}

\dedicatory{}

\begin{abstract} 
For $\alpha >0$, the $\alpha$-Lipschitz minorant of a function $f : \mathbb{R} \rightarrow \mathbb{R}$ is the greatest function $m : \mathbb{R} \rightarrow \mathbb{R}$ such that $m \leq f$ and $\vert m(s) - m(t) \vert \leq \alpha \vert s-t \vert$ for all $s,t \in \mathbb{R}$, should such a function exist. If $X=(X_t)_{t \in \mathbb{R}}$ is a real-valued L\'evy process that is not a pure linear drift with slope $\pm \alpha$, then the sample paths of $X$ have an $\alpha$-Lipschitz minorant almost surely if and only if $\mathbb{E}[\vert X_1 \vert]< \infty$ and $\vert \mathbb{E}[X_1]\vert < \alpha$. Denoting the minorant by $M$, we consider the contact set $\mathcal{Z}:=\{ t \in \mathbb{R} : M_t = X_t \wedge X_{t-}\}$, which, since it is regenerative and stationary, has the distribution of the closed range of some subordinator ``made stationary'' in a suitable sense.  We provide a description of the excursions of the L\'evy process away from its contact set similar to the one presented in It\^o excursion theory. We study the distribution of the excursion on the special interval straddling zero.  We also give an explicit path decomposition of the other ``generic'' excursions in the case of  Brownian motion with drift $\beta$ with $\vert \beta \vert < \alpha$.  Finally, we investigate the progressive enlargement of the Brownian filtration by the random time that is the first point of the contact set after zero.
\end{abstract}

\maketitle

\section{Introduction}

Recall that a function $g: \reals \to \reals$ is $\alpha$-Lipschitz for some $\alpha > 0$ if $|g(s) - g(t)| \le \alpha |s - t|$ for all $s,t \in \reals$.  
Given a function $f : \reals \to \reals$, we say that $f$ {\em dominates} the $\alpha$-Lipschitz function $g$ if $g(t) \le f(t)$ for all $t \in \reals$.  
A necessary and sufficient condition that $f$ dominates some $\alpha$-Lipschitz function is that $f$ is bounded below on compact intervals and satisfies
$\liminf_{t \to -\infty} f(t) - \alpha t > - \infty$ and 
$\liminf_{t \to +\infty} f(t) + \alpha t > - \infty$.  
When the function $f$ dominates some $\alpha$-Lipschitz function there is an $\alpha$-Lipschitz function $m$ dominated by $f$ such that $g(t) \le m(t)$ for all $t \in \reals$ for any $\alpha$-Lipschitz function $g$ dominated by $f$; we call $m$ the {\em $\alpha$-Lipschitz minorant} of $f$.  The $\alpha$-Lipschitz minorant is given concretely by
\begin{equation}
\label{mformula}
\begin{split}
m(t) & = \sup \{ h \in \reals : h - \alpha|t-s| \leq f(s)  \text{ for all } s \in \reals \} \\
& = \inf \{f(s) + \alpha |t-s| : s \in \reals\}. \\
\end{split}
\end{equation}
The purpose of the present paper is to continue the study of the $\alpha$-Lipschitz minorants of the sample paths of a two-sided L\'evy process begun in \cite{zbMATH06288068}.

A two-sided L\'evy process is a real-valued stochastic process indexed by the real numbers that has c\`adl\`ag paths, stationary independent increments, and takes the value $0$ at time $0$.  
The distribution of a two-sided L\'evy process $X$ is characterized by the L\'evy-Khintchine formula
$\E[e^{i \theta (X_t - X_s)}] = e^{-(t-s) \Psi(\theta)}$ for $\theta \in \reals$ and $-\infty < s \le t < \infty$,
where
\[
\Psi(\theta) = - i a \theta + \frac{1}{2} \sigma^2 \theta^2 +
\int_\reals(1 - e^{i \theta x} + i \theta x \ind_{ \{ |x| \le 1 \} } )
\, \Pi(dx) 
\]
with $a \in \reals$, $\sigma \in \reals_+$, and $\Pi$ a $\sigma$-finite
measure concentrated on $\reals \setminus \{0\}$ 
satisfying $\int_\reals (1 \wedge x^2) \, \Pi(dx) <
\infty$ (see \cite{bertoin, sato} for information about (one-sided) L\'evy processes --- the two-sided case involves only trivial modifications).  
In order to avoid having to consider annoying, but trivial, special cases in what follows, we henceforth assume that $X$ is not just deterministic linear drift $X_t = a t$, $t \in \reals$, for some $a \in \reals$; that is, we assume that there is a non-trivial Brownian component ($\sigma > 0$) or a non-trivial jump component ($\Pi \ne 0$).

The sample paths of $X$ have bounded variation almost surely
if and only if $\sigma = 0$ and
$\int_\reals (1 \wedge |x| ) \, \Pi(dx) < \infty$. In this case $\Psi$
can be rewritten as
\[
\Psi(\theta) = - i d \theta + \int_\reals (1-e^{i \theta x} ) \, \Pi(dx).
\]
We call $d \in \reals$ the drift coefficient.

We now recall a few facts about the $\alpha$-Lipschitz minorants of the sample paths of $X$ from \cite{zbMATH06288068}.

Either the $\alpha$-Lipschitz minorant exists for almost all sample paths of $X$ or it fails to exist for almost sample paths of $X$.
A necessary and sufficient condition for the $\alpha$-Lipschitz minorant to exist for almost all sample paths is that 
$\mathbb{E}[\vert X_1 \vert]< \infty$ and $\vert \mathbb{E}[X_1]\vert < \alpha$.  
We assume from now on that this condition holds and denote the corresponding minorant process by $(M_t)_{t \in \reals}$.
Figure~\ref{F:brownian-picture} shows an example of a typical Brownian motion sample path and its associated $\alpha$-Lipschitz minorant.

\begin{figure}
  \centering
    \reflectbox{%
      \includegraphics[width=1.0\textwidth]{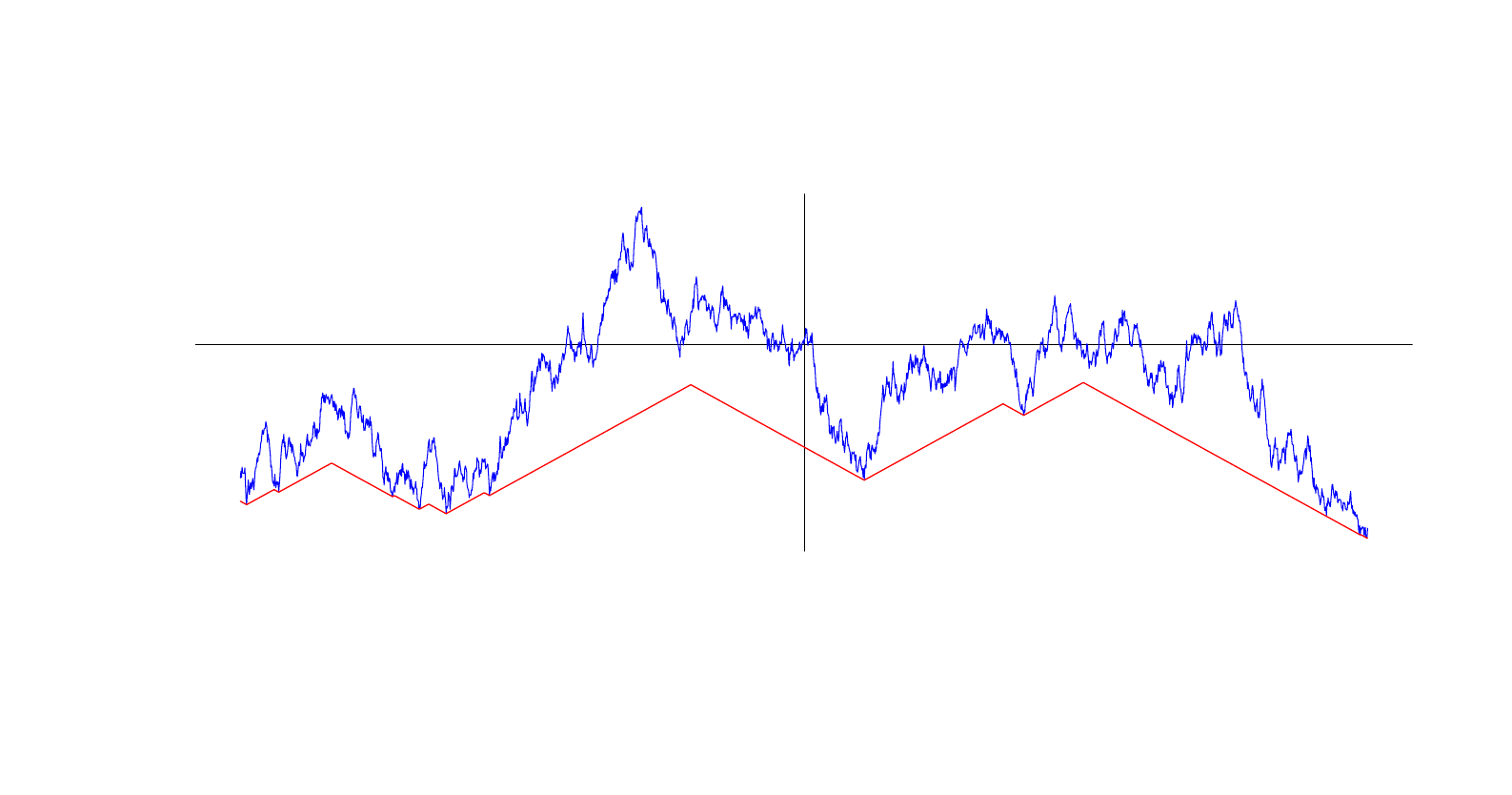}}
  \caption{A typical Brownian motion sample path and its associated
$\alpha$-Lipschitz minorant.}
\label{F:brownian-picture}
\end{figure}

Set $\mathcal{Z} := \{t \in \reals : M_t = X_t \wedge X_{t-}\}$.  We call $\mathcal{Z}$ the {\em contact set}.
The random closed set $\mathcal{Z}$ is non-empty, stationary, and regenerative in the sense of \cite{fitztaksar}
(see Definition~\ref{D:regenset} below for a re-statement of the definition).
Such a random closed set either has infinite Lebesgue measure almost surely or zero Lebesgue measure almost surely.
\begin{itemize}
\item
If the sample paths of $X$ have unbounded variation almost surely, then $\mathcal{Z}$ has zero Lebesgue measure almost surely.
\item
If $X$ has sample paths of bounded variation and $| d | > \alpha$, then $\mathcal{Z}$ has zero Lebesgue measure almost surely.
\item
If $X$ has sample paths of bounded variation and $| d | < \alpha$, then $\mathcal{Z}$ has infinite Lebesgue measure almost surely.
\item
If $X$ has sample paths of bounded variation and  $\vert d \vert = \alpha$, then whether the Lebesgue measure of $\mathcal{Z}$ is infinite or zero is determined  by an integral condition involving the L\'evy measure $\Pi$ that we omit.  
In particular, if $\sigma=0$, $\Pi(\mathbb{R}) < \infty$, and $\vert d \vert = \alpha$, then the Lebesgue measure of $\mathcal{Z}$ is almost surely infinite.
\end{itemize}


If $\mathcal{Z}$ has zero Lebesgue measure, then $\mathcal{Z}$ is either almost surely a discrete set or almost surely a perfect set with empty interior.
\begin{itemize}
\item
If $\sigma > 0$, then $\mathcal{Z}$ is almost surely discrete.
\item
If $\sigma = 0$ and $\Pi(\mathbb{R})=\infty$, then $\mathcal{Z}$ is almost surely discrete if and only if
\[
 \int_{0}^{1} t^{-1}\mathbb{P}\{X_t \in [-\alpha t,\alpha t]\} \,  dt < \infty.
\]
\item
If $\sigma = 0$ and $\Pi(\mathbb{R}) < \infty$, then $\mathcal{Z}$ is almost surely discrete if and only if $|d| > \alpha$.
\end{itemize}

The outline of the remainder of the paper is as follows.  

In Section~\ref{S:space-time_regenerative} we show that the pair $((X_t)_{t \in \reals}, \mathcal{Z})$ is a {\em space-time regenerative system} in the sense that if $D_t := \inf\{s \ge t: s \in \mathcal{Z}\}$ for any $t \in \reals$, then $((X_{D_t+u} - X_{D_t})_{u \ge 0}, \mathcal{Z} \cap [D_t,\infty) - D_t)$ is independent of $((X_u)_{u \le D_t}, \mathcal{Z} \cap (-\infty, D_t])$ with a distribution that does not depend on $t \in \reals$.  
It follows that if $\mathcal{Z}$ is discrete, we write 
$0 < T_1 < T_2 < \ldots$ for the successive positive elements of $\mathcal{Z}$, and we set
$Y^n = (X_{T_n + t} - X_{T_n}, \, 0 \le t \le T_{n+1} - T_n)$, $n \in \nats$, for the corresponding sequence of excursions away
from the contact set, then these excursions are independent and identically distributed.  When $\mathcal{Z}$ is not discrete there is a ``local time'' on $\mathcal{Z} \cap [0,\infty)$ and we give a description of the corresponding excursions away from the contact set as the points of a Poisson point process that is analogous to It\^o's description of the excursions of a Markov process away from a regular point.

Because $((X_t)_{t \in \reals}, \mathcal{Z})$ is stationary, the key to establishing the space-time regenerative property is to show that if $D := D_0$ is the first positive point in $\mathcal{Z}$, then $((X_{D+t} - X_D)_{t \ge 0}, \mathcal{Z} \cap [D,\infty) - D)$ is independent of $((X_t)_{t \le D}, \mathcal{Z} \cap (-\infty, D])$.  This is nontrivial because $D$ is most definitely not a stopping time for the canonical filtration of $X$ and so we can't just apply the strong Markov property.  We derive the claimed fact in Section~\ref{S:after_first_contact} using a result from \cite{millarpostmin} on the path decomposition of a real-valued Markov process at the time it achieves its global minimum.  This result in turn is based on general last-exit decompositions from \cite{MR0297019, MR0334335}.

When the contact set is discrete we obtain some information about the excursion away from the $\alpha$-Lipschitz minorant that contains the time zero in Section~\ref{S:straddle} using ideas from \cite{zbMATH01416816}.   If $G$ is the last contact time before zero and $D$, as above, is the first contact time after zero, we show that $\frac{D}{D - G}$ is independent of $(X_t - X_G, \, G \le t < D)$ and uniformly distributed on $[0,1]$.   This observation allows us to describe the finite-dimensional distributions of $(X_t, \, G \le t < D)$ in terms of those of $(X_t, \, 0 \le t < D)$, and we are able to determine the latter explicitly.   The argument here is based on a generalization of the fact that if $V$ is a nonnegative random variable, $U$ is uniformly distributed on $[0,1]$, and $U$ and $V$ are independent, then it is possible to express the distribution of $V$ in terms of that of $U V$.

As before, write $Y_n$, $n \in \nats$, for the independent, identically distributed sequence of excursions away from the contact set that occur at positive times in the case where the contact set is discrete.   When $X$ is Brownian motion with drift $\beta$, where $|\beta| < \alpha$ in order for the $\alpha$-Lipschitz minorant to exist, we establish a path decomposition description for the common distribution of the $Y_n$ in Section~\ref{S:Brownian_generic}.  Using this path decomposition we can determine the distributions of quantities such as the length $T_{n+1} - T_n$ and the distribution of the final value
$X_{T_{n+1}} - X_{T_n}$.  Moreover, if we write $Y_0$ for the excursion straddling time zero, then we have the ``size-biasing'' relationship
$\mathbb{E}[f(Y_0)] = \mathbb{E}[f(Y_n) (T_{n+1} - T_n)] / \mathbb{E}[T_{n+1} - T_n]$, $n \in \nats$, for nonnegative measurable functions $f$, and this allows us to recover information about the distribution of $Y_0$ from a knowledge of the common distribution of the ``generic'' excursions $Y_n$, $n \in \nats$.

As we noted above, the random time $D$ is not a stopping time for the canonical filtration of $X$.  In Section~\ref{S:enlargement} we investigate the filtration obtained by enlarging the Brownian filtration in such way that $D$ becomes a stopping time.  Martingales for the Brownian filtration become semimartingales in the enlarged filtration and we are able to describe their canonical semimartingale decompositions quite explicitly.

The paper finishes with two auxiliary sections.  Section~\ref{S:minorants} contains some (deterministic) results about the $\alpha$-Lipschitz minorant construction that are used throughout the paper.  Section~\ref{S:stopping_time}
details two general lemmas about random times for L\'evy processes that are used in Section~\ref{S:after_first_contact} and  Section~\ref{S:enlargement}.

\section{Space-time regenerative systems}
\label{S:space-time_regenerative}

Let $\Omega^\lra$ (resp. $\Omega^\ra$) denote the space of c\`adl\`ag 
$\reals$-valued paths indexed by $\reals$ (resp. $\reals_+$).
For $t \in \reals$, define $\tau_t : \Omega^\lra \to \Omega^\ra$ by
\[
(\tau_t(\omega^\lra))_s := \omega^\lra_{t+s} - \omega^\lra_t, \quad s \ge 0.
\]
For $t \in \reals$ define $x_t : \Omega^\lra \to \reals$ by
\[
x_t(\omega^\lra) := \omega^\lra_t.
\]
For $t \in \reals$, define $k_t : \Omega^\lra \to \Omega^\lra$ by

\[
(k_t(\omega^\lra))_s:=
\begin{cases}
 \omega^\lra_s,&  \text{if } s \le t, \\
  \omega^\lra_t,& \text{if } s>t.
\end{cases}
\]
Let $\tilde \Omega^\lra$ (resp. $\tilde \Omega^\ra$)
denote the class of closed subsets of $\reals$ (resp. $\reals_+$). For $t
\in \reals$ define $\tilde \tau_t: \tilde \Omega^\lra \to \tilde \Omega^\ra$
by 
\[
\tilde \tau_t(\tilde \omega^\lra) := \{s-t: s \in \tilde \omega^\lra \cap [t, \infty)\}.
\]
For $t \in \reals$ define $d_t: \tilde \Omega^\lra \to \reals \cup \{+\infty\}$ by
\[
d_t(\tilde \omega^\lra) := \inf \{ s>t: s \in \tilde \omega^\lra \}
\]
and 
$r_t: \tilde \Omega^\lra \to \reals_+ \cup \{+\infty\}$ by
\[
r_t(\tilde \omega^\lra) := d_t(\tilde \omega^\lra) - t.
\]
With a slight abuse of notation, also use $d_t$ and $r_t$, $t \in \reals_+$, to denote the analogously
defined maps from $\tilde \Omega^\ra$ to $\reals_+ \cup \{+\infty\}$.

Put $\bar \Omega^\lra := \Omega^\lra \times \tilde \Omega^\lra$ and
$\bar \Omega^\ra := \Omega^\ra \times \tilde \Omega^\ra$.
Define $\bar \tau_t : \bar \Omega^\lra \to \bar \Omega^\ra$ by
\[
\bar \tau_t(\omega^\lra, \tilde \omega^\lra) 
:= (\tau_t(\omega^\lra), \tilde \tau_t(\tilde \omega^\lra)).
\]
Define $\bar d_t : \bar \Omega^\lra \to \reals \cup \{+\infty\}$ by
\[
\bar d_t(\omega^\lra, \tilde \omega^\lra) := d_t(\tilde \omega^\lra).
\]
Finally, for $t \in \reals$ define the following $\sigma$-fields on $\bar \Omega^\lra$:
\[
\bar{\mathcal{G}_t}^\lra:=\sigma \{ \bar d_s , k_{\bar d_s}, s \le t \}
\]
 and 
 \[
 \bar{\mathcal{G}}^\lra:=\sigma \{ \bar d_s , k_{\bar d_s}, s \in \reals \}.
 \]
Define $\bar{\mathcal{G}_t}^\ra$ and $\bar{\mathcal{G}}^\ra$ analogously.

\begin{defn}
\label{D:regenset}
Let $\bar \Q^\lra$ (resp. $\bar \Q^\ra$) be a
probability measure on $(\bar \Omega^\lra,\bar{\mathcal{G}}^\lra)$
(resp. $(\bar \Omega^\ra,\bar{\mathcal{G}}^\ra)$.  Then $\bar \Q^\lra$
is regenerative with regeneration law $\bar \Q^\ra$ if
\begin{itemize}
  \item[(i)] 
	$\bar \Q^\lra \{\bar d_t = +\infty\} = 0$, for all $t \in \reals$;
  \item[(ii)]
  for all $t \in \reals$ and
    for all $\bar{\mathcal{G}}^\ra$-measurable nonnegative functions $F$,
\[
\bar \Q^\lra \left [F(\bar \tau_{\bar d_t}) \, | \, \bar{\mathcal{G}}_{t+} \right] = \bar \Q^\ra[F], 
\]
where we write $\bar \Q^\lra[\cdot]$ and $\bar \Q^\ra[\cdot]$ for expectations
with respect to $\bar \Q^\lra$ and $\bar \Q^\ra$.
\end{itemize}
\end{defn}

\begin{remark}
\label{R:zero_enough}
Suppose that the probability measure $\bar \Q^\lra$ on $(\bar \Omega^\lra, \bar{\mathcal{G}}^\lra)$
is stationary; that is, that under $\bar \Q^\lra$ the process 
$(\omega^\lra, \tilde \omega^\lra) \mapsto (x_t(\omega^\lra), r_t(\tilde \omega^\lra))_{t \in \reals}$
has the same distribution as the process
$(\omega^\lra, \tilde \omega^\lra) \mapsto (x_{s+t}(\omega^\lra)-x_{s}(\omega^\lra), r_{s+t}(\tilde \omega^\lra))_{t \in \reals}$
for all $s \in \reals$.
Then,
in order to check conditions (i) and (ii) of Definition~\ref{D:regenset},
it suffices to check them for the case $t=0$.
\end{remark}

\begin{thm}
\label{T:extension}
\begin{itemize}
\item[(i)]
In order to check that the probability measure $\bar \Q^\lra$ on $(\bar \Omega^\lra, \bar{\mathcal{G}}^\lra)$
is space-time regenerative with the probability measure $\bar \Q^\ra$ on $(\bar \Omega^\ra, \bar{\mathcal{G}}^\ra)$
as regeneration law, it suffices to check
\begin{itemize}
  \item[(a)] 
	$\bar \Q^\lra \{\bar d_t = +\infty\} = 0$, for all $t \in \reals$;
  \item[(b)] 
  for all $t \in \reals$ and
    for all $\bar{\mathcal{G}}^\ra$-measurable nonnegative functions $F$,
\[
\bar \Q^\lra \left [F(\bar \tau_{\bar d_t}) \, | \, \bar{\mathcal{G}}_{t} \right] = \bar \Q^\ra[F].
\]
\end{itemize}
\item[(ii)]
Suppose that the probability measure $\bar \Q^\lra$ on $(\bar \Omega^\lra, \bar{\mathcal{G}}^\lra)$
is space-time regenerative with the probability measure $\bar \Q^\ra$ on $(\bar \Omega^\ra, \bar{\mathcal{G}}^\ra)$
as regeneration law and that $T$ is an almost surely finite 
$(\bar{\mathcal{G}}^\lra_{t+})_{t \in \reals}$-stopping
time.  Then for all $\bar{\mathcal{G}}^\ra$-measurable nonnegative functions $F$
\[
\bar \Q^\lra \left [F(\bar \tau_{\bar d_T}) \, | \, \bar{\mathcal{G}}_{T+} \right] = \bar \Q^\ra[F].
\]
\end{itemize}
\end{thm}

\begin{proof}
(i) Fix $t \in \reals$. For $n \in \nats$ set $t_n  := t + 2^{-n}$.

Consider $F : \bar \Omega^\ra \to \reals_+$ of the form
\[
F((\omega^\ra, \tilde \omega^\ra)) 
= f(\omega^\ra_{s_1}, \ldots, \omega^\ra_{s_\ell}, r_{s_1}(\tilde \omega^\ra), \ldots, r_{s_\ell}(\tilde \omega^\ra))
\]
for some $0 \le s_1 < s_2 < \ldots < s_\ell$ and bounded, continuous
function $f:\reals^\ell \times (\reals \cup \{+\infty\})^\ell \to \reals_+$.
For such an $F$ we have 
\[
\lim_{n \to \infty} F(\bar \tau_{\bar d_{t_n}}(\bar \omega^\lra)) 
= F(\bar \tau_{\bar d_t}(\bar \omega^\lra))
\]
for all $\bar \omega^\lra \in \bar \Omega^\lra$
and it suffices by a monotone class argument to show that
\[
\bar \Q^\lra \left [F(\bar \tau_{\bar d_{t_n}}) \, | \, \bar{\mathcal{G}}^\lra_{t+} \right] = \bar \Q^\ra[F]
\]
for all $n \in \nats$.  This, however, is clear because
$\bar{\mathcal{G}}^\lra_{t+} \subseteq \bar{\mathcal{G}}^\lra_{t_n}$ and 
\[
\bar \Q^\lra \left [F(\bar \tau_{\bar d_{t_n}}) \, | \, \bar{\mathcal{G}}^\lra_{t_n} \right] = \bar \Q^\ra[F]
\]
by assumption.

\noindent
(ii) For $n \in \nats$ define a $(\bar{\mathcal{G}}^\lra_t)_{t \in \reals}$-stopping time
$T_n$ by declaring that $T_n := \frac{k}{2^n}$ when $T \in [\frac{k-1}{2^n}, \frac{k}{2^n})$,
$k \in \ints$.  

Let $F$ be as in the proof of part (i).
For such an $F$ we have 
\[
\lim_{n \to \infty} F(\bar \tau_{\bar d_{T_n}}(\bar \omega^\lra))
= F(\bar \tau_{\bar d_T}(\bar \omega^\lra))
\]
for all $\bar \omega^\lra \in \bar \Omega^\lra$
and it suffices by a monotone class argument to show that
\[
\bar \Q^\lra \left [F(\bar \tau_{\bar d_{T_n}}) \, | \, \bar{\mathcal{G}}^\lra_{T+} \right] = \bar \Q^\ra[F]
\]
for all $n \in \nats$.
Since $\bar{\mathcal{G}}_{T+} \subseteq \bar{\mathcal{G}}_{T_n+}$ for all $n \in \nats$,
it further suffices to show that
\[
\bar \Q^\lra \left [F(\bar \tau_{\bar d_{T_n}}) \, | \, \bar{\mathcal{G}}^\lra_{T_n+} \right] = \bar \Q^\ra[F].
\]

Fix $n \in \nats$ and suppose that $G$ is a nonnegative  $\bar{\mathcal{G}}^\lra_{T_n+}$-measurable random variable.
We have
\[
\begin{split}
\bar \Q^\lra \left [F(\bar \tau_{\bar d_{T_n}})\,  G \right] 
& =
\sum_{k \in \ints}
\bar \Q^\lra \left [F(\bar \tau_{\bar d_{T_n}}) \, G \, \ind\left\{T_n = \frac{k}{2^n}\right\}\right] \\
&  =
\sum_{k \in \ints}
\bar \Q^\lra \left [F(\bar \tau_{\bar d_{\frac{k}{2^n}}}) \, G  \, \ind\left\{T_n = \frac{k}{2^n}\right\}\right] \\
&  = \bar \Q^\ra[F] \sum_{k \in \ints}
\bar \Q^\lra \left [G  \, \ind\left\{T_n = \frac{k}{2^n}\right\}\right] \\
&  = \bar \Q^\ra[F] \,  
\bar \Q^\lra \left [G \right], \\
\end{split}
\]
where in the penultimate equality we used the fact that $G \, \ind\{T_n = \frac{k}{2^n}\}$ is $\bar{\mathcal{G}}^\lra_{\frac{k}{2^n}}$-measurable
(see, for example, \cite[Lemma 7.1(ii)]{MR1876169}).  This completes the proof.
\end{proof}

\begin{thm}
\label{T:contact_regenerative}
Suppose for the L\'evy process $(X_t + \alpha t)_{t \in \reals}$ that $0$ is regular for $(0,\infty)$.
Then the distribution of $((X_t)_{t \in \reals}, \mathcal{Z})$ is space-time regenerative.
\end{thm}

\begin{proof}
Use Theorem~\ref{T:post-D} below, Remark~\ref{R:zero_enough}, and part (i) of Theorem~\ref{T:extension}.
\end{proof}

\begin{remark}
\label{R:time-reversal}
If $0$ is not regular for $(0,\infty)$ for the L\'evy process $(X_t + \alpha t)_{t \in \reals}$, then $0$ is regular for $(-\infty, 0)$ for the L\'evy process $(X_t - \alpha t)_{t \in \reals}$.  
Equivalently, if $0$ is not regular for $(0,\infty)$ for the L\'evy process $(X_t + \alpha t)_{t \in \reals}$, then $0$ is regular for $(0,\infty)$ for the L\'evy process $(-X_t + \alpha t)_{t \in \reals}$ and hence for the L\'evy process $(X_{-t-}+\alpha t)_{t \in \reals}$.  
Thus, either the distribution of $((X_t)_{t \in \reals}, \mathcal{Z})$ is space-time regenerative or the distribution of $((X_{-t-})_{t \in \reals}, \mathcal{Z})$ is space-time regenerative.
\end{remark}

Write $\tilde \pi^\lra: \bar \Omega^\lra \to \tilde \Omega^\lra$ for the projection $\bar \omega^\lra = (\omega^\lra, \tilde \omega^\lra) \mapsto \tilde \omega^\lra$.
Define $\tilde \pi^\ra: \bar \Omega^\ra \to \tilde \Omega^\ra$ similarly.
If $\bar \Q^\lra$ is space-time regenerative with regeneration law $\bar \Q^\ra$, then, in the sense of \cite{fitztaksar}, the push-forward
of $\bar \Q^\lra$ by $\tilde \pi^\lra$ is regenerative with regeneration law the push-forward of $\bar \Q^\ra$ by $\tilde \pi^\ra$.  It follows that 
$\bar \Q^\ra\{(\omega^\ra, \tilde \omega^\ra) : \tilde \omega^\ra \; \text{is discrete}\}$ is either $1$ or $0$.  

Suppose that the probability in question is $1$.
Define $(\bar{\mathcal{G}}^\lra_{t})_{t \in \reals}$-stopping times $T_1, T_2, \ldots$ with $0 < T_1 < T_2 < \ldots$ almost surely by
\[
T_1 := \bar d_0
\]
and
\[
T_{n+1}(\bar \omega^\lra) := \bar d_{T_n(\bar \omega^\lra)}(\bar \omega^\lra) = T_n(\bar \omega^\lra) + \bar d_0  \circ \bar \theta_{T_n(\bar \omega^\lra)} (\bar \omega^\lra), \quad n \in \nats,
\]
where $\bar \theta_t: \bar \Omega^\lra \to \bar \Omega^\lra$, $t \in \reals$, are the shift maps given by
$\bar \theta(\omega^\lra, \tilde \omega^\lra) =((\omega^\lra_{t+u})_{u \in \reals}, \tilde \omega^\lra - t)$.
Let  $\partial$ be an isolated cemetery state adjoined to $\reals$.
Define c\`adl\`ag $\reals \cup \{\partial\}$-valued processes $Y^n = (Y^n_t)_{t \in \reals_+}$, $n \in \nats$, by
\[
Y^n_t(\omega^\lra, \tilde \omega^\lra)
:=
\begin{cases}
\pi^\ra\circ \tau_{\bar T_n(\omega^\lra, \tilde \omega^\lra)}(\omega^\lra, \tilde \omega^\lra)_t,& 0 \le t < \bar d_0  \circ \bar \theta_{T_n(\bar \omega^\lra)} (\bar \omega^\lra)=\zeta_n, \\
\partial,& t \ge \bar d_0  \circ \bar \theta_{T_n(\bar \omega^\lra)} (\bar \omega^\lra)=\zeta_{n}, \\
\end{cases}
\]
where $\pi^\ra : \bar \Omega^\ra \to \Omega^\ra$ is the projection $(\omega^\ra, \tilde \omega^\ra) \mapsto \omega^\ra$.
Then, under $\bar \Q^\lra$, the sequence $Y^n$, $n \in \nats$, is independent and identically distributed.


The path of of each $Y_n$ lies in the set $\Omega^{0,\partial}$ consisting of c\`adl\`ag functions $f:\reals_+ \to \reals \cup \{\partial\}$ such that $f(0) = 0$, $0 < \inf\{s \ge 0: f(s) = \partial\} < \infty$, and $f(t) = \partial$ for all $t \ge \inf\{s \ge 0: f(s) = \partial\}$.

When the probability in question is $0$ there is a local time on our regenerative set and we can construct a Poisson random measure on the set $\reals \times \Omega^{0,\partial}$
that records the excursions away from the contact set and the order in which they occur.  We use the following theorem which is a restatement of \cite[Corollary~3.1]{zbMATH03665914}.\\

\begin{thm} 
\label{T:nested_arrays}
Let $(O_k)_{k \in \nats}$ be an increasing family of measurable sets in a measurable space $(O,\mathcal{O})$ such that $O = \bigcup_{k \in \nats} O_k$.
Let $\mathbf{V}$ be an $O$-valued point process; that is, $\mathbf{V} = (\mathbf{V}_t)_{t \ge 0}$ is a stochastic process with values in $O \cup \{\dag\}$ for some adjoined point $\dag$ such that $\{t \ge 0: \mathbf{V}_t \ne \dag\}$ is almost surely countable. 
Suppose that $\{t \ge 0 : \mathbf{V}_t \in O_k\}$ is almost surely discrete and unbounded for all $k \in \nats$ while $\{t \ge 0 : \mathbf{V}_t \in O\}$ is almost surely not discrete. Suppose that the sequence $\{\mathbf{V}_t : \mathbf{V}_t \in O_k\}$ is independent and identically distributed for each $k \in \nats$. 
For $k \in \nats$ define  $(N_k(t))_{t \ge 0}$ by setting $N_k(t) = \#\{0 \le u \le t : \mathbf{V}_u \in O_k\}$ for $t \ge 0$.
 For $k \in \nats$ set $T_k = \inf\{t > 0 : \mathbf{V}_t \in O_k\}$ and put $p_k = \P\{\mathbf{V}_{T_k} \in O_1\}$. 
Then, for almost all $\omega \in \Omega$, uniformly for bounded $t \ge 0$,
\[
\lim_{k \rightarrow \infty} p_k N_k(t,\omega)=L(t,\omega),
\]
where $L(t,\omega)$ is continuous and nondecreasing in $t \ge 0$, and strictly increasing on $\{t \ge 0: \mathbf{V}(t,\omega) \neq \dag \}$. 
For such $\omega$, set
\[
\mathbf{V}^{*}(s,\omega)
= 
\begin{cases}
 V(t,\omega),&  \text{if } s=L(t,\omega), \\
  \dag,& \text{otherwise.}
\end{cases}
\]
Then $\mathbf{V}^{*}$ is a homogeneous Poisson point process; that is, the random measure that puts mass $1$ at each point $(t,f) \in \reals_+ \times O$ such that $\mathbf{V}^*_t = f$ is a Poisson random measure with intensity of the form $\lambda \otimes \nu$, where $\lambda$ is Lebesgue measure on $\reals_+$ and $\nu$ is a $\sigma$-finite measure on $(O,\mathcal{O})$.  
Moreover, for almost all $\omega \in \Omega$ for all $t$ with $\mathbf{V}(t,\omega) \ne \dag$, $\mathbf{V}(t,\omega) = \mathbf{V}^{*}(L(t,\omega), \omega)$.
\end{thm}

We can apply this theorem if we consider $O$ to be the space $\Omega^{0,\partial}$ of c\`adl\`ag paths that vanish at the origin and have finite lifetimes, and take $O_k$ to be the subspace of paths with lifetime at least $\frac{1}{k}$. 
We define $\mathbf{V}$ to be the point process of the excursions such that, for every $t \ge 0$, $\mathbf{V}_t$ is equal to the excursion whose right end point is $t$, with the convention that $\mathbf{V}_t = \dag$ if $t$ is not the right end point of an excursion. In the case where $\mathcal{Z}$ is not discrete, all the conditions of Theorem~\ref{T:nested_arrays} can readily be checked and we obtain a time-changed Poisson point process.

\section{The process after the first positive point in the contact set}
\label{S:after_first_contact}

\begin{notation}
\label{N:G_D}
For $t \in \reals$ set $G_t :=\sup (\mathcal{Z} \cap (-\infty,t))$  and $D_t : =\inf (\mathcal{Z} \cap (t,+\infty))$. 
Put $G := G_0$ and $D := D_0$.
\end{notation}

\begin{remark}
\label{R:S_construction}
We have from Lemma~\ref{L:recipe}  that
\[
D = \inf\{t  \ge S : X_t \wedge X_{t-}+\alpha t = \inf \{X_u +\alpha u : u \geq S \}\},
\]
where
\[
S = S_0 := \inf\{s>0 : X_s \wedge X_{s-}-\alpha s \leq \inf\{X_u -\alpha u : u \leq 0 \} \}
\]
because almost surely $X_{S} \le X_{S-}$. The latter result was shown in the proof of \cite[Theorem~2.6]{zbMATH06288068}.
\end{remark}

\begin{notation} 
For $t \in \reals$ put $\mathcal{F}_t  = \bigcap_{\epsilon > 0} \sigma\{X_s : -\infty < s \le t + \epsilon\}$.
Define the $\sigma$-field $\mathcal{F}_{U}$ for any nonnegative random time $U$ to be the $\sigma$-field generated by all the random variables of the form $\xi_U$ where $(\xi_t)_{t \in \mathbb{R}}$ is an optional process with respect to the filtration $(\mathcal{F}_{t})_{t \in \reals}$.  Similarly, define $\mathcal{F}_{U-}$ to be the $\sigma$-field generated by all the random variables of the form $\xi_U$ where $(\xi_t)_{t \in \mathbb{R}}$ is now a previsible process with respect to the filtration $(\mathcal{F}_{t})_{t \in \reals}$
\end{notation}

\begin{notation}
Let $\tilde{X} = (\tilde \Omega, \tilde{\mathcal{F}},  \tilde{\mathcal{F}}_t, \tilde X_t, \tilde \theta_t, \tilde{\mathbb{P}}^x)$ be a Hunt process such that the distribution of $\tilde{X}$ under $\tilde{\mathbb{P}}^x$ is that of $(x + X_t + \alpha t)_{t \ge 0}$.  Put $\breve T:=\inf\{ t > 0: \tilde{X}_s \wedge \tilde{X}_{s-} <0\}$, and for $t \ge 0$ and $x,y > 0$ put
\[
\breve{H}_t(x,dy) 
:= 
\tilde{\mathbb{P}}^x\{\tilde{X}_t \in dy , t < \breve T\}
\frac{\tilde{\mathbb{P}}^y\{\breve T=\infty\}}{\tilde{\mathbb{P}}^x\{\breve T=\infty\}}.
\]
We interpret $(\breve{H}_t)_{t \ge 0}$ as the transition functions of the Markov process $\tilde X$ conditioned to stay positive.
\end{notation}

\begin{thm}
\label{T:post-D}
Suppose that $0$ is regular for $(0,\infty)$ for the Markov process $\tilde X$. 
Then the process $(X_{t+D}-X_D)_{t \geq 0}$ is independent of $(X_t ,  -\infty < t \le D)$.  
Moreover, the process $(X_{t+D}-X_D + \alpha t)_{t \geq 0}$ is Markovian with transition functions $(\breve H_t)_{t \ge 0}$
and a certain family of entrance laws $(\breve Q_t)_{t \ge 0}$.
\end{thm}

\begin{proof} 
Because $(X_t)_{t \in \reals}$ is a two-sided L\'evy process and $S$ is a stopping time, the process $\check X := (X_{t+S}-X_{S}+\alpha t)_{t \geq 0}$ is, by the strong Markov property, independent of $\mathcal{F}_{S}$  and has the same distribution as the process $\tilde X$ under $\tilde{\mathbb{P}}^0$. 

By \cite[Proposition~2.4]{millarzeroone}, the set $\{t \ge 0: \tilde{X}_t \wedge \tilde{X}_{t-} = \inf\{\tilde{X}_s : s \ge 0\}\}$ consists $\tilde{\mathbb{P}}^x$-almost surely of a single point $\tilde T$ for all $x \in \reals$.
Consequently, the set $\{t \ge 0: \check{X}_t \wedge \check{X}_{t-} = \inf\{\check{X}_s : s \ge 0\}\}$ also consists almost surely of a single point $\check T$.
From Remark~\ref{R:S_construction} we have $D=S+\check{T}$.

Because $0$ is regular for $(0,\infty)$ for the Markov process $\tilde X$ and thus $\tilde{X}_{\tilde T}=\inf \{\tilde{X}_s : s \ge 0\}$, it follows from the sole theorem in \cite{millarpostmin} that the process $(\tilde X_{\tilde T + t})_{t \ge 0}$ is independent of $\tilde{\mathcal{F}}_{\tilde T}$ given $\tilde{X}_{\tilde{T}}$.

Moreover, there exists a family of entrance laws $(Q_t(x;\cdot))_{t \ge 0}$ for each $x \in \reals$ and a family of transition functions $(H_t(x; \cdot, \cdot)_{t \ge 0}$ for each $x \in \reals$ such that
\[
\tilde{\mathbb{P}}^{x}\{\tilde{X}_{t+\tilde{T}} \in A \, \vert \, \tilde{\mathcal{F}}_{\tilde{T}}\}
=
Q_t(\tilde{X}_{\tilde{T}}; A), \quad t \ge 0,
\]
and
\[
\tilde{\mathbb{P}}^{x}\{\tilde{X}_{t+\tilde{T}} \in A \, \vert \, \tilde{\mathcal{F}}_{\tilde{T}+s}\}
=
H_{t-s}(\tilde{X}_{\tilde{T}}; \tilde{X}_{\tilde{T}+s}, A), \quad 0<s<t.
\]
Using the fact that the processes $(x + \tilde{X}_{t+\tilde{T}})_{t\geq 0}$ under $\tilde{\mathbb{P}}^{0}$ and $(\tilde{X}_{t+\tilde{T}})_{t\geq 0}$ under $\tilde{\mathbb{P}}^{x}$ have the same law,
%
it follows that $Q_t(x; x+A)=Q_t(0,A)$ and $H_t(x; x+y, x+A)=H_t(0; y,A)$.
Thus the process $(\tilde{X}_{t+\tilde{T}}-\tilde{X}_{\tilde{T}})_{t \geq 0}$ is independent of $\mathcal{\tilde{F}}_{\tilde T}$ and, moreover, this process is Markovian with the entrance law $A \mapsto Q_t(0; A) =: \breve Q_t(A)$, $t \ge 0$ and transition functions $(y,A) \mapsto H_t(0;y,A) = \breve H_t(y,A)$, $t \ge 0$.
Applying Lemma ~\ref{optional:filtration} we get that $(X_{t+D}-X_D + \alpha t)_{t \geq 0}$ = $(\check X_{t + \check T} - \check X_{\check T})_{t \ge 0}$ is independent of $\mathcal{F}_{D-} \vee \sigma \{X_D\}$ and Markovian with transition functions $(\breve H_t)_{t \ge 0}$
and entrance laws $(\breve Q_t)_{t \ge 0}$. 

Introduce the killed process $(\bar{X})_{t \in \reals}$ defined  by
\[
\bar{X}_t :=
\begin{cases}
        X_t, &   t < D, \\
       X_D,& t = D, \\
         \partial,  &  t > D, \\       
\end{cases}
\]
where $\partial$ is an adjoined isolated point.
To complete the proof, it suffices to show that
$\sigma\{X_t, -\infty < t \le D\} \equiv \sigma\{\bar{X}_t, t \in \reals\} \subseteq \mathcal{F}_{D-} \vee \sigma \{X_D\}$; that is,
that $\bar X_t$ is $\mathcal{F}_{D-} \vee \sigma \{X_D\}$-measurable for all $t \in \reals$.
For all $u \in \reals$ the process $(\ind_{\{s > u\}})_{s \in \reals}$ is left-continuous, right-limited, and $(\mathcal{F}_s)_{s \in \reals}$-adapted.  Therefore, for all $u \in \reals$, the random variable $\ind_{\{D > u \}}$ is $\mathcal{F}_{D-}$-measurable and so the random variable $D$ is $\mathcal{F}_{D-}$-measurable.  In particular,  the event $\{ t >D \}$ is $\mathcal{F}_{D-}$-measurable.  Next, the process $(X_t \ind_{t < s})_{s \in \reals}$ is also left-continuous, right-limited, and $(\mathcal{F}_s)_{s \in \reals}$-adapted and hence the random variable $X_t\ind_{\{ t < D\}}$ is $\mathcal{F}_{D-}$-measurable.  Consequently, for any Borel subset $A \subseteq \reals \cup \{\partial\}$ we have
\[
\begin{split}
& \{\bar X_t \in A\} \\
& \quad = (\{\bar X_t \in A\} \cap \{t < D\}) \cup (\{\bar X_t \in A\} \cap \{t = D\}) \cup (\{\bar X_t \in A\} \cap \{t  > D\}) \\
& \quad = 
\begin{cases}
(\{X_t\ind_{\{ t < D\}} \in A\} \cap \{t < D\}) \cup ( \{X_D \in A\} \cap \{t = D\}) \cup \{t > D\}, & \partial \in A,\\
(\{X_t\ind_{\{ t < D\}} \in A\} \cap \{t < D\}) \cup ( \{X_D \in A\} \cap \{t = D\}), & \partial \notin A,\\
\end{cases} \\
& \quad \in
\mathcal{F}_{D-} \vee \sigma \{X_D\}, \\
\end{split}
\]
as claimed.
\end{proof}

\section{The excursion straddling zero}
\label{S:straddle}

In this section we focus on the excursion away from the contact set that straddles the time zero; that is, the piece of the path of $X$ between the times $G$ and $D$ of Notation~\ref{N:G_D}.

The following proposition gives an explicit path decomposition for, and hence the distribution of, the process $(X_u, \, 0 \le u \le D)$.

\begin{prop}
Set
\begin{equation*}
I^{-}:=\inf \{X_u-\alpha u : u \le 0\}.
\end{equation*}
Consider the following independent random objects :
\begin{itemize}
\item
a random variable $\Gamma$ with the same distribution as $I^{-}$,
\item
$(X'_t)_{t \ge 0}$ and $(X''_t)_{t \ge 0}$ two independent copies of $(X_t)_{t \ge 0}$.
\end{itemize}
Define the process $(Z_t)_{t \ge 0}$ by
\[
Z_t :=
\begin{cases}
        X'_t, &  0 \le t \le T'_{\Gamma}, \\
         X''_{t-T'_{\Gamma}}+X'_{T'_{\Gamma}},  & T'_{\Gamma} \le t \le  T'_{\Gamma}+\tilde{T''}, \\
         \partial, &  t > T'_{\Gamma}+\tilde{T''} ,\\
\end{cases}
\]
where
\begin{equation*}
T'_{\Gamma}:=\inf \{ t \ge 0 : X'_t \wedge X'_{t-}-\alpha t \le \Gamma \}
 \end{equation*} 
and
 \begin{equation*}
 \tilde{T''}:=\inf \{ t \ge 0 : X''_t \wedge X''_{t-}+\alpha t = \inf \{ X''_u+ \alpha u : u \ge 0 \}\}.
 \end{equation*}
 Then,
 \begin{equation*}
 (X_t , \, 0 \le t \le D) \ed (Z_t , \, 0 \le t \le T'_{\Gamma}+\tilde{T''}).
 \end{equation*}
\end{prop}

\begin{proof}
The path decomposition follows from the construction of the points $S$ and $D$ in Remark~\ref{R:S_construction}. 
The proof is left to the reader.
\end{proof}

Now that we have the distribution of the path of $X$ on $[0,D]$, let us extend it to the whole interval $[G,D]$.
First of all, we will prove that the random variable $U=\frac{-G}{D-G}$ is independent from the straddling excursion $(X_{t+G}-X_{G}, \, 0 \leq t \leq D-G)$ and has a uniform distribution on the interval $(0,1]$.

Our approach here uses ideas from \cite[Chapter~8]{zbMATH01416816} but with a modification of the particular shift operator 
considered there; 
see also \cite{MR1957884} for a framework with general shift operators that encompasses the setting we work in.  There is a large literature
in this area of general Palm theory that is surveyed in \cite{zbMATH01416816, MR1957884} but we mention \cite{MR0474493, MR941992}
as being of particular relevance.

We will prove general results for the path space $(H,\mathcal{H})$ and sequence space $(L,\mathcal{L})$ defined by 
\[ 
H :=\{ (z_t)_{t \in \reals} : z \text{ is real-valued and c\`adl\`ag with } z(0)=0 \}
\]
and
\[
L :=\{ (s_k)_{k \in \ints} \in \mathbb{R}^{\mathbb{Z}} : -\infty \leftarrow \dots <s_{-1} < 0 \leq s_0 < s_1 < \dots \rightarrow \infty \}.
\]
We take $\mathcal{H}$ to be the $\sigma$-field on $H$ that makes all of the maps $z \mapsto z_t$, $t \in \reals$, measurable, and $\mathcal{L}$ to be the trace of the product $\sigma$-field on $L$.


For $t \in \reals$ define the {\em shift} $\theta_t: H \times L \to H \times L$ by 
\[
\theta_t((z_s)_{s \in \reals},(s_k)_{k \in \ints})=((z_{t+s}-z_t)_{s \in \reals}, (s_{n_t+k}-t)_ {k \in \ints})
\]
where $n_t=n$ if and only if $t\in [s_{n-1},s_n)$.  The family $(\theta_t)_{t \in \reals}$ is measurable in the sense that the mapping 
\[
((z_s)_{s \in \reals},(s_k)_{k \in \ints},t) \in H \times L \times \reals \mapsto \theta_t((z_s)_{s \in \reals},(s_k)_{k \in \ints}) \in H \times L
\]
 is $\mathcal{H} \otimes \mathcal{L} \otimes \mathcal{B} / \mathcal{H} \otimes \mathcal{L}$ measurable, where $\mathcal{B}$ is the Borel $\sigma$-field on $\reals$.

We consider a probability space $(\Omega, \mathcal{F}, \mathbb{P})$ equipped with a random pair $(K,P)$ that take values on $H \times L$. 
We assume furthermore that $(K,P)$ is {\em space-homogeneous stationary} in the sense that
\[
\theta_s(K,P)\ed (K,P), \text{ for all } s\in \reals.
\] 
\begin{remark}
When $\mathcal{Z}$ is discrete, the space-time regenerative system $((X_t)_{t \in \reals}, \mathcal{Z})$ is obviously space-homogeneous stationarity  due to the two facts that for any $s \in \reals$ we have $(X_{t+s}-X_s)_{t \in \reals} \ed (X_t)_{t \in \reals}$
and that the contact set for $(X_{t+s}-X_s)_{t \in \reals}$ is, by Lemma~\ref{L:space-time_homogeneous}, just
$\mathcal{Z}-s$.
\end{remark}

\begin{defn}
\begin{itemize}

\item{
Write $l_n$ for the $n^{\mathrm{th}}$ cycle length defined by $l_n=P_n-P_{n-1}$.}
\item
For $t \in \reals$, put $N_t=n$ for $t \in [P_{n-1},P_n)$.
\item
Define the relative position of $t$ in $[P_{N_t-1},P_{N_t})$ by  $ U_t :=\frac{t-P_{N_t-1}}{P_{N_t}}$.
\item
Define the random variable $(K^{\circ},P^{\circ})$ by 
\[
(K^{\circ},P^{\circ})=\theta_{P_0}(K,P)=((K_{t+P_0}-K_{P_0})_{t \in \reals}, (P_{k}-P_0)_{k \in \ints}).
\]
\end{itemize}
\end{defn}

The following are two important features of the family $(\theta_t)_{t \in \reals}$ that are useful in proving results analogous to those in  \cite[Chapter~8, Section~3]{zbMATH01416816}.

\begin{prop} 
\label{P:shift_properties}
The family of shifts $(\theta_t)_{t \in \reals}$ enjoys the two following properties.
\begin{itemize}
\item[(i)]
 The family $(\theta_t)_{t \in \reals}$ is semigroup; that is, for every $t,s  \in \reals : \theta_t \circ \theta_s=\theta_{t+s}$.
\item[(ii)]
For all $s \in \reals$ and $(K,P) \in H \times L$ we have $\theta_s(K,P)^{\circ}=\theta_{N_s}(K,P)$.
\end{itemize}
\end{prop}

\begin{proof}
For all $t,s \in \reals$, and $(K,P)\in H\times L$
\begin{equation*}
\begin{split}
\mathrm{Proj}_{H}[(K_t \circ \theta_s(K,P))]_u&=\theta_s(K)_{u+t}-\theta_s(K)_t\\
&=(K_{u+t+s}-K_s)-(K_{t+s}-K_{s})\\
&=K_{u+t+s}-K_{t+s}\\
&=(\theta_{t+s}(K))_u\\
&=\mathrm{Proj}_{H}[\theta_{t+s}(K,P)]_u
\end{split}
\end{equation*}
where $\mathrm{Proj}_H$ is the projection from $H\times L$ to $H$. 
The proof for the action of the shift on the sequence component is given in \cite[Chapter 8, Section 2]{zbMATH01416816}. 

We prove the (ii) in a similar manner.  We have
\begin{equation*}
\begin{split}
\theta_s(K,P)^{\circ}&=((\theta_s(K)_{t+\theta_s(P)_0}-\theta_s(K)_{\theta_s(P)_0})_{t \in \reals}, (P_{N_s+k}-P_{N_s})_
{k \in \ints})\\
&=(((K_{t+s+\theta_s(P)_0}-K_{s})-(K_{\theta_s(P)_0+s}-K_s))_{t\in \reals},(P_{N_s+k}-P_{N_s})_{k \in \ints})\\
&=((K_{t+s+P_{N_s}-s}-K_{s+P_{N_s}-s})_{t\in \reals}, (P_{N_s+k}-P_{N_s})_{k \in \ints})\\
&=((K_{t+P_{N_s}}-K_{P_{N_s}},(P_{N_s+k}-P_{N_s})_{k \in \ints})\\
&=\theta_{P_{N_s}}(K,P)
\end{split}
\end{equation*}
\end{proof}

We state now a theorem that is analogous to parts of \cite[Chapter~8, Theorem~3.1]{zbMATH01416816}.  
The proof uses the same key ideas as that result and just exploits the two properties of the family of shifts laid out in Proposition~\ref{P:shift_properties}.

\begin{thm} 
\label{T:UO_uniform_independent}
The random variable $U_0$ is uniform on $[0,1)$ and is independent of $(K^{\circ},P^{\circ})$.  
Also,
\[
\mathbb{E}\left[ \sum_{k=1}^{N_t} f(\theta_{P_k}(K,P))\right]=t\mathbb{E}\left[\frac{f(K^{\circ},P^{\circ})}{l_0}\right].
\]
\end{thm}

\begin{proof} 	Consider a nonnegative Borel function $g$ on $[0,1)$ and a nonnegative $\mathcal{H}\otimes \mathcal{L}$-measurable function $f$. To establish both claims of the theorem, it suffices to prove that 
\begin{equation}
\label{U_0_indept}
t\mathbb{E}\left[g(U_0)\frac{f(K^{\circ},P^{\circ})}{l_0}\right]
=\left( \int_{0}^1 g(x) \, dx \right)\mathbb{E}\left[\sum_{k=1}^{N_t} f(\theta_{P_k}(K,P))\right].
\end{equation}

By stationarity, the left-hand side of  equation \eqref{U_0_indept} is 
\[
t\mathbb{E}\left[g(U_0)\frac{f(K^{\circ},P^{\circ})}{l_0}\right]
= \int_{0}^t \mathbb{E}\left[\theta_s\left[g(U_0)\frac{f(K^{\circ},P^{\circ})}{l_0}\right]\right] \, ds
\]

Because $\theta_s(U_0)=U_s$ and $\theta_s(l_0)=l_{N_s}$, and using the fact that $\theta_s(K,P)^{\circ}=\theta_{N_s}(K,P)$.
We have
\begin{equation*}
\begin{split}
t\mathbb{E}\left[g(U_0)\frac{f(K^{\circ},P^{\circ})}{X_0}\right
]&= \int_{0}^t \mathbb{E}\left[g(U_s)\frac{f(\theta_{P_{N_s}}(K,P)}{l_{N_s}}\right] \, ds\\
&=\mathbb{E}\left[ \sum_{k=1}^{N_t} \int_{P_{k-1}}^{P_k} \frac{g(U_s)f(\theta_{P_{N_s}}(K,P))}{l_{N_s}}\, ds\right] \\
& \quad +\mathbb{E} \left[ \int_{0}^{P_0}g(U_s)\frac{f(\theta_{P_{N_s}}(K,P)}{l_{N_s}}ds \right] \\
& \quad -\mathbb{E} \left[ \int_{t}^{P_{N_t}}g(U_s)\frac{f(\theta_{P_{N_s}}(K,P)}{l_{N_s}}ds \right] \\
&=\mathbb{E}\left[ \sum_{k=1}^{N_t} f(\theta_k(K,P))\int_{P_{k-1}}^{P_k}\frac{g(U_s)}{l_k}\, ds\right].\\
\end{split}
\end{equation*}
It follows from stationarity that
\[
\mathbb{E} \left[ \int_{0}^{P_0}g(U_s)\frac{f(\theta_{P_{N_s}}(K,P)}{l_{N_s}}ds \right]=\mathbb{E} \left[ \int_{t}^{P_{N_t}}g(U_s)\frac{f(\theta_{P_{N_s}}(K,P)}{l_{N_s}}ds \right].
\]
A change of variable in the integral shows that
\[
 \int_{P_{k-1}}^{P_k}\frac{g(U_s)}{l_k} \, ds=\int_{0}^{l_k} \frac{g(\frac{s}{l_k})}{l_k} \, ds
=
\int_{0}^1 g(x) \, dx, 
\]
and this proves the claim \eqref{U_0_indept}.
\end{proof}

The next result is the analogue of \cite[Chapter 8, Theorem 4.1]{zbMATH01416816}. 

\begin{cor}
\label{C:cycle_stationarity}
For any nonnegative $\mathcal{H}\otimes \mathcal{L}$- measurable function $f$ and every $n \in \ints$, we have
\[
\mathbb{E}\left[\frac{f(\theta_{P_n}(K,P))}{l_0}\right]=\mathbb{E}\left[\frac{f(K^{\circ},P^{\circ})}{l_0}\right].
\]
\end{cor}

\begin{proof}
It suffices to consider the case when $f$ is bounded by a constant $A$. 
Applying Theorem~\ref{U_0_indept} with the function $f$ replaced by $f\circ \theta_{P_n}$, we have, for all $t\ge 0$,
\begin{align*}
t\mathbb{E}\left[\frac{f(\theta_{P_n}(K,P))}{l_0}\right]
&=\mathbb{E}\left[\sum_{k=1}^{N_t} f(\theta_{P_{k+n}}(K,P))\right]\\
&=\mathbb{E}\left[\sum_{k=1}^{N_t} f(\theta_{P_{k}}(K,P))\right]-\mathbb{E}\left[\sum_{k=1}^{n} f(\theta_{P_{k}}(K,P))\right]\\
&\quad+\mathbb{E}\left[\sum_{k=N_t+1}^{N_t+n} f(\theta_{P_{k}}(K,P))\right]\\
&=t\mathbb{E}\left[\frac{f(K^{\circ},P^{\circ})}{l_0}\right]-\mathbb{E}\left[\sum_{k=1}^{n} f(\theta_{P_{k}}(K,P))\right]\\
&\quad +\mathbb{E}\left[\sum_{k=N_t+1}^{N_t+n} f(\theta_{P_{k}}(K,P))\right].\\
\end{align*}
Hence
\[ 
\left\vert \mathbb{E}\left[\frac{f(\theta_{P_n}(K,P))}{l_0}\right
]-\mathbb{E}\left[\frac{f(K^{\circ},P^{\circ})}{l_0}\right] \right\vert \le \frac{A n}{t}.
\]
Letting $t \rightarrow \infty$ finishes the proof.
\end{proof}

Of particular interest to us in our L\'evy process setting is the case where the sequence 
$((K_{t+P_{n-1}} - K_{P_{n-1}}, \, 0 \le t < l_n), \, l_n)_{n \in \ints}$ is independent, in which  case the sequence
$((K_{t+P_{n-1}} - K_{P_{n-1}}, \, 0 \le t < l_n), \, l_n)_{n \ne 0}$ 
is independent and identically distributed (cf. \cite[Chapter 8, Remark 4.1]{zbMATH01416816}).
Part (i) of the following result is a straightforward consequence of Corollary~\ref{C:cycle_stationarity}.  
Part (ii) is immediate from part (i).
We omit the proofs.

\begin{cor}
\label{C:size_biasing}
Suppose that the sequence  $((K_{t+P_{n-1}} - K_{P_{n-1}}, \, 0 \le t < l_n), \, l_n)_{n \in \ints}$ is independent
\begin{itemize}
\item[(i)]
For a nonnegative measurable function $f$,
\[
\begin{split}
& \mathbb{E}\left[f((K_{t+P_{0}} - K_{P_{0}}, \, 0 \le t < l_1), \, l_1)\right] \\
& \quad =
\mathbb{E}\left[\frac{1}{l_0}\right]^{-1} 
\mathbb{E}\left[f((K_{t+P_{-1}} - K_{P_{-1}}, \, 0 \le t < l_0), \, l_0) \frac{1}{l_0}\right]. \\
\end{split}
\]
\item[(ii)]
For a nonnegative measurable function $f$,
\[
\begin{split}
& \mathbb{E}\left[f((K_{t+P_{-1}} - K_{P_{-1}}, \, 0 \le t < l_0), \, l_0) \right] \\
& \quad =
\mathbb{E}[l_1]^{-1}
\mathbb{E}\left[f((K_{t+P_{0}} - K_{P_{0}}, \, 0 \le t < l_1), \, l_1) l_1 \right]. \\
\end{split}
\]
\end{itemize}
\end{cor}

We return to our L\'evy process  set-up and assume that $\mathcal{Z}$ is discrete.
The pair $((X_t)_{t \in \reals}, \mathcal{Z})$ is space-homogeneous stationary and hence, by Theorem~\ref{T:UO_uniform_independent}, the random variable $U_0=\frac{-G}{D-G}$ is uniform on $[0,1)$ and independent of the process $(X_{t+D}-X_{D})_{t \in \reals}$.  Put
\[
V_t
:=
\begin{cases}
X_{D}-X_{(D-t)-}, & 0 \le t < D-G =: \zeta_V, \\
\partial, & t \ge D-G. \\
\end{cases}
\]
It is easy to check that $D-G$ is the first positive point of the contact set of the process $(X_D-X_{(D-t)-})_{t \in \reals}$ and so the process $(V_t)_{t \in \reals}$ can be written as 
\begin{equation*}
V=F(X_{D+\cdot}-X_{D}),
\end{equation*} 
where $F$ is a measurable function from the space of c\`adl\`ag functions on the real line to the space of c\`adl\`ag functions on the positive real line. Hence the random variable $U=1- U_0$ is independent of $(V_t)_{t \ge 0}$.  We have already observed that we know the distribution of the process
\[
W_t
:=
\begin{cases}
V_t, & 0 \le t < D=U(D-G) =: \zeta_W, \\
\partial, & t \ge D=U(D-G). \\
\end{cases}
\]
We now show that it is possible to derive the distribution of $V$ from that of $W$.

\begin{cor}
\label{C:extension}
Recall that  $\zeta_V$ (resp. $\zeta_W$) the lifetime of the process $(V_t)_{t \ge 0}$ (resp. $(W_t)_{t \geq 0}$).
For bounded, measurable functions $f_1, \ldots, f_n$ that take the value $0$ at $\partial$ and times
$0 \le t_1 < \ldots < t_n < t< \infty$,
\[
\begin{split}
\E[f_1(V_{t_1}) \cdots f_n(V_{t_n}) \, \ind\{t <  \zeta_V\}]
& =
\E[f_1(W_{t_1}) \cdots f_n(W_{t_n}) \, \ind\{t < \zeta_W\}] \\
& \quad +
t \frac{\E[f_1(W_{t_1}) \cdots f_n(W_{t_n}) \, \ind\{\zeta_W \in dt\}]}{dt}. \\
\end{split}
\]
\end{cor}

\begin{proof}
Observe that
\[
\begin{split}
& \E[f_1(W_{t_1}) \cdots f_n(W_{t_n}) \, \ind\{t < \zeta_W\}] \\
& \quad =
\E[f_1(V_{t_1}) \cdots f_n(V_{t_n}) \, \ind\{t < U \zeta_V\}] \\
& \quad =
\int_0^1 \E[f_1(V_{t_1}) \cdots f_n(V_{t_n}) \, \ind\{t / u <  \zeta_V\}] \, du \\
& \quad =
\int_1^\infty \E[f_1(V_{t_1}) \cdots f_n(V_{t_n}) \, \ind\{t s <  \zeta_V\}] \frac{1}{s^2} \, ds \\
& \quad =
\int_t^\infty \E[f_1(V_{t_1}) \cdots f_n(V_{t_n}) \, \ind\{r <  \zeta_V\}] \frac{t^2}{r^2} \frac{1}{t} \, dr \\
& \quad =
t \int_t^\infty \E[f_1(V_{t_1}) \cdots f_n(V_{t_n}) \, \ind\{r <  \zeta_V\}] \frac{1}{r^2} \, dr, \\
\end{split}
\]
so that
\[
\int_t^\infty \E[f_1(V_{t_1}) \cdots f_n(V_{t_n}) \, \ind\{r <  \zeta_V\}] \frac{1}{r^2} \, dr
=
\frac{1}{t} \E[f_1(W_{t_1}) \cdots f_n(W_{t_n}) \, \ind\{t < \zeta_W\}].
\]
Differentiating both sides with respect to $t$ and rearranging gives the result.
\end{proof}

\begin{remark}
(i) The proof of Corollary~\ref{C:extension} is similar to that of \cite [Appendix, Proposition 3.12]{fa071d831f55425d9228e8affec59ee6} which gives an analytic link between the distributions of $Z$ and $UZ$ where $Z$ and $U$ are independent nonnegative random variables with $U$ uniform on $[0,1]$.

\noindent
(ii) We gave above a way to find the distribution of the excursion straddling zero. To determine the distribution of $(X_t, \, G \le t \le D)$ we generate the process $V$ according to the distribution described above with lifetime $\zeta_{V}$, and then take an independent random variable $U$ uniform on $[0,1]$, then we have the equality of distributions 
\begin{equation*}
  (X_t , \, G \le t \le D) \ed (V_{t+U\zeta_V}-V_{U\zeta_V}, \, -U\zeta_V \le t \le (1-U)\zeta_V).
\end{equation*}

\noindent
(iii) As a particular consequence of the Theorem~\ref{T:UO_uniform_independent}, we can find the distribution of the 
straddling excursion length $D-G$
if we know the distribution of the right-hand endpoint $D$.    See \cite[Remark 8.2]{zbMATH06288068} where the relevant calculations are carried out to find the Laplace transform  $D-G$.
\end{remark}

From now on, we distinguish between the generic excursions (that is, all the excursions that start after time $D$ or finish before $G$). These excursions are independent and identically distributed and independent of the excursion straddling zero between $G$ and $D$. In the next section we give a description of the common distribution of the generic excursions in the case of the Brownian motion with drift. 

\section{A generic excursion for Brownian motion with drift}
\label{S:Brownian_generic}


Suppose in this section that $X$ is two-sided Brownian motion with drift $\beta$ such that $\vert \beta \vert < \alpha$; that is, $X = (B_t+\beta t)_{t \in \reals}$, where $B$ is a standard linear Brownian motion.

We recall the Williams path decomposition for Brownian motion with drift (see, for example, \cite[Chapter VI, Theorem~55.9]{rogerswills}).

\begin{thm}
\label{Williams}
Let $\mu>0$ on some probability space, take three independent random elements:
\begin{itemize}
\item
$( B^{(-\mu)}_t, \,  t \geq 0)$ a BM with drift $-\mu$;
\item
$(R^{(\mu)}_t,  t \geq 0)$ a diffusion that is solution of the following SDE 
\[ 
dR^{(\mu)}_t=dB_t+\mu \coth(\mu R^{(\mu)}_t)dt , \quad R^{(\mu)}_0=0,
 \]
 where $B$ is a standard Brownian motion;
\item
$\gamma$ an exponential r.v with rate $2\mu$.
\end{itemize}
Set $\tau=\inf\{ t  \ge 0 : B^{(-\mu)}_t=-\gamma \} $ and 
\[
H_t=
\begin{cases}
        B^{(-\mu)}_t,& 0\leq t\leq \tau, \\
        R^{(\mu)}_{t-\tau}-\gamma,& t \ge \tau.\\
\end{cases}
\]
Then, $(H_t)_{t \ge 0}$ is a Brownian motion with drift $\mu$.
\end{thm}

\begin{remark}
The diffusion $R^{(\mu)}$ is called a 3-dimensional Bessel process with drift $\mu$ denoted $\mathrm{BES}(3,\mu)$. We may use a superscript to refer to the starting position of this process, when there is no superscript it implicitly means we start at zero.
This process has the same distribution as the radial part of a 3-dimensional Brownian motion with drift of magnitude $\mu$
\cite[Section 3]{zbMATH03731090}. 
This process may be thought of as a Brownian motion with drift $\mu$ conditioned to stay positive. 
\end{remark}

We give some results about Bessel processes that will be useful later in our proofs. The first result is a last exit decomposition of a Bessel process presented in \cite[Chapter~6, Proposition~3.9]{zbMATH02150787}.

\begin{prop}
\label{Bessel:last exit}
 Let $\rho$ be $\mathrm{BES}^{x}(3)$; that is, $\rho$ is a 3-dimensional Bessel process started at $x \ge 0$.
Let $T$ be a stopping time with respect to the filtration $\mathcal{F}^{\rho, J} :=(\sigma \{ \rho_s, J_s, 0 \le s \leq t \})_{t \ge 0}$, where $J_t:=\inf_{s \le t} \rho_s$ is the future infinimum of $\rho$.
 Then ($\rho_{T+t}-\rho_{T})_{t \ge 0}$ is a $\mathrm{BES}^0(3)$ that is independent of $(\rho_t, \, 0 \le t \le T)$. 
In particular if 
 \begin{equation*}
 g_{x,y}:=\sup \{ t \ge 0 : \rho_t=y \}=\inf \{ t \ge 0 : \rho_t=J_t=y\}, \quad y \ge x,
 \end{equation*}
 then $(\rho_{t+g_{x,y}}-y)_{t \ge 0}$ is a $\mathrm{BES}^0(3)$ independent of $(\rho_t, \, 0 \le t \le g_{x,y})$.
\end{prop}

The next result relates the time-reversed Bessel process and the Brownian motion. It is from 
\cite[Chapter 7, Corollary 4.6]{zbMATH02150787}

\begin{prop}
\label{time:reversal}
Let $b>0$, $\rho$ be a $\mathrm{BES}^0(3)$, and $B$ be a standard linear Brownian motion. We have the equality of distributions
\begin{equation*}
(\rho_{L_b-t}, \, 0\le t \le L_b) \ed (b-B_t, \, 0 \le t \le T_b),
\end{equation*}
where $L_b:=\sup \{ t \ge 0: \rho_t=b \}$ is the last passage time of $\rho$ at the level $b$ and $T_b:=\inf \{ t \ge 0 : B_t=b\}$ is the first hitting time of the Brownian motion $B$ started at zero to $b$. In particular,
\[
L_b \ed T_b. 
\]
\end{prop}

The final result we will need is a path decomposition of a 3-dimensional Bessel process with drift started at a positive initial state when it hits its ultimate minimum. We don't know a reference for this result, so we give its proof for the sake of completeness.

\begin{thm}
\label{Bessel:decomp}
Let $b, \mu >0$. Consider the following three independent random elements :
\begin{itemize}
\item
a random variable $g$ with density proportional to $e^{2\mu x}$ supported on $[0,b]$;
\item
a Brownian motion $(B^{(b,-\mu)}_t)_{t \ge 0}$ with drift $-\mu$ started at $b$;
\item
a 3-dimensional Bessel process $(R^{(\mu)}_t)_{t \ge 0}$ with drift $\mu$ started at zero.
\end{itemize}
where $\tilde{T}_g:=\inf \{ t \ge 0 : B^{(b,-\mu)}_t=g \}$.
\[
R^{(b,\mu)}_t=
\begin{cases}

        B^{(b,-\mu)}_t,& 0\leq t\leq \tilde{T}_g, \\
        g+R^{(\mu)}_{t-\tilde{T}_g}, & t \ge \tilde{T}_g. \\
\end{cases}
\]
Then, $R^{(b,\mu)} \ed \mathrm{BES}^b(3,\mu)$;  that is, $R^{(b,\mu)}$ is a 3-dimensional Bessel process with drift $\mu$ started at $b$.
\end{thm}

\begin{proof}
The distribution of a 3-dimensional Bessel process with drift $\mu$ and started at $b>0$ is the conditional distribution of a Brownian motion with drift $\mu$ started at $b$ conditioned to stay positive (see the Remarks at the end of \cite[Section 3]{zbMATH03731090}). The event we condition on has a positive probability, so it is just the usual naive conditioning
\[
(b-\mathrm{BM}^0(-\mu)) \; \bigg \vert  \;  \left\{\sup_{t \ge 0} \mathrm{BM}^0(-\mu)_t \le b\right\} \ed \mathrm{BES}^b(3,\mu),
\]
where $\mathrm{BM}^0(-\mu)$ is a Brownian motion with drift $-\mu$ and started at zero. 
The theorem is then just an application of the Williams path decomposition Theorem~\ref{Williams}.
\end{proof}

Recall that in this section $(X_t)_{t \in \reals}$ is a Brownian motion with drift $\beta$. The discussion in Theorem ~\ref{T:post-D} and the Williams path decomposition Theorem~\ref{Williams} shows that $(X_{t+D}-X_{D}+\alpha t)_{t \ge 0}$ has the same distribution as $(B_t+(\alpha+\beta)_t)_{t \in \reals}$ conditioned to stay positive. Thus,
\begin{equation*}
(X_{t+D}-X_D , \, t \ge 0) = (R^{(\alpha+\beta)}_t-\alpha t, \, t \ge 0),
\end{equation*} 
where $R^{(\alpha+\beta)} \ed \text{BES}(3,\alpha+\beta)$.
We aim now to provide a path decomposition of the first positive generic excursion away from the contact set (and thus all generic excursions), that is the path of $(W_t)_{t\ge 0}:=(X_{t+D}-X_D)_{t \ge 0} = (X_{t+D_0}-X_{D_0})_{t \ge 0}$ until it hits the first contact point $D_{D_0}-D_0$. 

\begin{notation}
Using Lemma ~\ref{L:simplified recip}, let us define the following times that are the analogues of $\textbf{s}$ and $\textbf{d}$ for this generic excursion.
\begin{equation*}
\mathfrak{T}:=\inf \{t >0 : W_t-\alpha t \leq 0\}=\inf \left\{ t >0 : \frac{R^{(\alpha+\beta)}_t}{t}=2\alpha \right\}
\end{equation*}
and 
\begin{equation*}
\begin{split}
\zeta & := \inf \{ t \ge \mathfrak{T} : W_t+\alpha t =\inf \{W_u +\alpha u : u \ge \mathfrak{T} \} \} \\
& =\inf \{ t \ge \mathfrak{T} : R^{(\alpha+\beta)}_t=\inf \{R^{(\alpha+\beta)}_u : u \ge \mathfrak{T} \} \}.
\end{split}
\end{equation*}
\end{notation}

The following theorem is a path decomposition of a generic excursion away from the contact set.

\begin{thm}
\label{lipschitz:decomp}
Consider the following independent random elements:
\begin{itemize}
\item
a pair of random variables $(\tau,\hat{\gamma})$ with the joint density 
\begin{equation}
\label{joint_density}
f_{\tau,\hat{\gamma}}(t,x)=\frac{\exp \left( -\frac{(\alpha-\beta)^2t}{2}-2(\alpha+\beta)x\right)}{\sqrt{2\pi t^3}}\ind_{0 \le x \le 2\alpha t} , \quad t >0,
\end{equation}
\item
a standard Brownian excursion $\mathbf{e}$ on $[0,1]$,
\item
a linear Brownian motion $(\tilde{B}^{-(\alpha+\beta)}_t)_{t \ge 0}$ with drift $-(\alpha+\beta)$.
\end{itemize}
Define the process
\[
\mathfrak{E}_t=
\begin{cases}
        \sqrt{\tau}\mathbf{e}(\frac{t}{\tau}) +2\alpha t,&  0\leq t\leq \tau,\\
        2\alpha \tau + \tilde{B}^{-(\alpha+\beta)}_{t-\tau},&  \tau \le t \le \tau+\tilde{T}_{\hat{\gamma}}, \\
\end{cases}
\]
where $\tilde{T}_{\hat{\gamma}}:=\inf \{ t \ge 0 : \tilde{B}^{-(\alpha+\beta)}_t = -\hat{\gamma}\}$. 
Then,
\[
(X_{t+D}-X_D+\alpha t , \, 0 \le t \le \zeta ) \ed (\mathfrak{E}_t , \, 0 \le t \le \tau+\tilde{T}_{\hat{\gamma}}).
\]
\end{thm}

\begin{proof}
Let us first find first the distribution of the path of $R^{(\alpha+\beta)}$ on $[0,\mathfrak{T}]$. As $\mathfrak{T}$ is a stopping time (with respect to the filtration generated by $R^{(\alpha+\beta)}$) and $R^{(\alpha+\beta)}$ is a time-homogenous strong Markov process, conditioning on the value of $\mathfrak{T}$ (and thus on $R^{(\alpha+\beta)}_{\mathfrak{T}}=2\alpha \mathfrak{T}$ is enough to have the independence between the two components of our path).  
Define $(Y_t)_{t > 0}$ by
\begin{equation*}
Y_t:=tR^{(\alpha+\beta)}_{\frac{1}{t}}, \quad t > 0.
\end{equation*}
By the time-inversion property of  Brownian motion, $Y$ is a $\mathrm{BES}^{\alpha+\beta}(3)$; that is, $Y$ is a $3$-dimensional Bessel process started at $\alpha+\beta$ (with no drift).  
The stopping time $\mathfrak{T}$ can be expressed as
\begin{equation}
\label{last time}
\mathfrak{T}=\frac{1}{\sup \{ t \ge 0 : Y_t=2\alpha \}} \ed \frac{1}{g_{\alpha+\beta,2\alpha}}
\end{equation}

Hence by applying Proposition ~\ref{Bessel:last exit}
to our process $Y$ we find that
\begin{equation*}
(G_t, \, t \ge 0):=(Y_{t+\frac{1}{\mathfrak{T}}}-2\alpha , \, t \ge 0)
\end{equation*}
is a $\mathrm{BES}^0(3)$ independent from $\sigma \{ Y_u : u \le \frac{1}{\mathfrak{T}} \} = \sigma \{ R^{(\alpha+\beta)}_u :u \ge \mathfrak{T} \}$. Now, conditionally on $\{\mathfrak{T} = T\}$, we have :
\begin{align*}
(R^{(\alpha+\beta)}_u, \, 0\le u \le T)&=(u(G_{\frac{1}{u}-\frac{1}{T}}+2\alpha), \, 0 \le u \le T)\\
&=(uG_{\frac{T-u}{uT}}+2\alpha u , \, 0 \le u \le T).
\end{align*}
However, it is known that $(uG_{\frac{T-u}{uT}}, \, 0 \le u \le T)$ is just a Brownian excursion of length $T$ (that is a 3-dimensional Bessel bridge between $(0,0)$ and $(T,0)$).  This can easily be seen from the same time transformation that maps Brownian motions to Brownian bridges). For a reference to this path transformation,  see [page 226]\cite{MR733673}. Hence, given $\{\mathfrak{T} = T\}$, 
\begin{equation*}
(W_u , \, 0 \le u \le T) = (\mathbf{e}_T(u) + \alpha u , 0 \le u \le T) \ed \left(\sqrt{T}\mathbf{e}\left(\frac{u}{T}\right)+\alpha u , \, 0 \le u \le T\right),
\end{equation*}
where $\mathbf{e}_T$ is a Brownian excursion on $[0,T]$, and $\mathbf{e}$ is a standard Brownian excursion on $[0,1]$ obtained by Brownian scaling.  \\


Now lets move to the second fragment of our path; that is, the process $W$ on $[\mathfrak{T},\zeta]$. Because of the fact that $\mathfrak{T}$ is a stopping time and $R^{(\alpha+\beta)}$ is a strong Markov process, conditionally on $\{\mathfrak{T}=T\}$,  the process $(R^{(\alpha+\beta)}_{t+\mathfrak{T}}, \, 0 \le t \le \zeta-\mathfrak{T})$ is just a $\mathrm{BES}^{2\alpha T}(3,\alpha+\beta)$ stopped at the time it hits its ultimate minimum. Hence, by applying Theorem \ref{Bessel:decomp},
\begin{equation*}
(R^{(\alpha+\beta)}_{t+\mathfrak{T}}, \, 0 \le t \le \zeta-\mathfrak{T}) \ed (\tilde{B}^{(2\alpha T,-(\alpha+\beta))}_t , \, 0 \le t \le \tilde{T}_{\gamma}),
\end{equation*}
where $\tilde{B}^{(2\alpha T, -(\alpha+\beta))}$ is a standard Brownian motion with drift $-(\alpha+\beta)$ started at $2\alpha T$ and $\tilde{T}_{\gamma}:=\inf \{ t \ge 0: \tilde{B}^{(2\alpha T, -(\alpha+\beta))}_t=\gamma \}$, and $\gamma$ is independent of $\tilde{B}^{(2\alpha T,-(\alpha+\beta))}$ with density on $[0,2\alpha T]$ proportional to $x \mapsto e^{2(\alpha+\beta)x}$. 
Finally by setting $\hat{\gamma}=2\alpha \mathfrak{T}-\gamma$, it suffices to prove that $(\mathfrak{T},\hat{\gamma})$ has the joint density in \eqref{joint_density} to finish our proof.

We know that the conditional density of $\hat{\gamma}$ given $\{\mathfrak{T}=t\}$ is proportional to $x \mapsto e^{-2(\alpha+\beta) x}$ restricted to $[0,2\alpha t]$. That is,
\begin{equation}
\label{density:gamma}
f_{\hat{\gamma} \vert \mathfrak{T}=t }(x)=\frac{2(\alpha+\beta)e^{-2(\alpha+\beta)x}}{1-e^{-4(\alpha+\beta)\alpha t}}\ind_{ 0 \le x \le 2\alpha t}.
\end{equation}
To finish, let us find the distribution of $\mathfrak{T}$. Recall from \eqref{last time}
that we have
\[
\mathfrak{T} \ed \frac{1}{g_{\alpha+\beta,2\alpha}}.
\]
Now $g_{\alpha+\beta,2\alpha}$ is the last time a 3-dimensional Bessel process started at $\alpha+\beta$ visits the state $2\alpha$. Consider $(\tilde{Y}_t)_{t \ge 0}$ a $\mathrm{BES}^0(3)$, and let $H_{\alpha+\beta}:=\inf \{ t \ge 0 : \tilde{Y}_t=\alpha +\beta\}$ be the first hitting time of $\alpha +\beta$. Then, by the strong Markov property at time $H_{\alpha+\beta}$, we have
\[
L_{2\alpha} \ed H_{\alpha+\beta} +g_{\alpha+\beta,2\alpha},
\]
where $g_{\alpha+\beta,2\alpha}$ and $H_{\alpha+\beta}$ are independent, and $L_{2\alpha}$ is the last time $\tilde{Y}$ visits $2\alpha$. Hence we get the Laplace transform of $g_{\alpha+\beta,2\alpha}$ is
\[
\mathbb{E}[\exp(-\lambda g_{\alpha+\beta,2\alpha})]=\frac{\mathbb{E}[\exp( -\lambda L_{2\alpha})]}{\mathbb{E}[\exp(-\lambda H_{\alpha+\beta})]}.
\]
Using Proposition~\ref{time:reversal}, we know with the same notation that $L_{2\alpha} \ed T_{2\alpha}$. Thus,
\[
\mathbb{E}[\exp( -\lambda T_{2\alpha})]=\exp(-2\alpha \sqrt{2\lambda}).
\]
On the other hand, we obtain the Laplace transform of $H_{\alpha+\beta}$ from \cite[equation 2.1.4, p463]{handbook}, namely,
\[
\mathbb{E}[\exp( -\lambda H_{\alpha+\beta})]=\frac{(\alpha+\beta)\sqrt{2\lambda}}{\sinh((\alpha+\beta)\sqrt{2\lambda})}.
\]
Thus,
\[
\mathbb{E}[\exp(-\lambda g_{\alpha+\beta,2\alpha})]=\frac{e^{-2\alpha\sqrt{2\lambda}}\sinh((\alpha+\beta)\sqrt{2\lambda})}{(\alpha+\beta)\sqrt{2\lambda}}.
\]
Inverting this Laplace transform, we get the density of $g_{\alpha+\beta,2\alpha}$; that is,
\[
f_{g_{\alpha+\beta,2\alpha}}(t)=\frac{e^{-\frac{(\alpha-\beta)^2}{2t}}-e^{-\frac{(3\alpha+\beta)^2}{2t}}}{2(\alpha+\beta)\sqrt{2\pi t}}.
\]
The density of $\mathfrak{T}$ is thus
\begin{equation}
\label{density:t}
f_{\mathfrak{T}}(t)=\frac{1}{t^2}f_{g_{\alpha+\beta,2\alpha}}\left(\frac{1}{t}\right)=\frac{e^{-\frac{(\alpha-\beta)^2t}{2}}-e^{-\frac{(3\alpha+\beta)^2t}{2}}}{2(\alpha+\beta)\sqrt{2\pi t^3}}\ind_{t>0}.
\end{equation}
Multiplying  the \eqref{density:t} and \eqref{density:gamma} gives the desired equality.
\end{proof}

Now we have an explicit path decomposition of a generic excursion and we know the expression of the $\alpha$-Lipschitz minorant on the same interval in terms of the locations of the excursion at its end-points using Lemma~\ref{sawtooth}.  It is interesting to identify the distributions of the most important features such as:
\begin{itemize}
\item
the lifetime $\zeta$  of the excursion;
\item
the time $L$ at which the $\alpha$-Lipschitz minorant of the excursion attains its maximal value;
\item
the final value $W_{\zeta}$ of the excursion
\end{itemize}
--- see Figure~\ref{F:excursion}.

\begin{figure}
  \centering
      \includegraphics[width=1.0\textwidth]{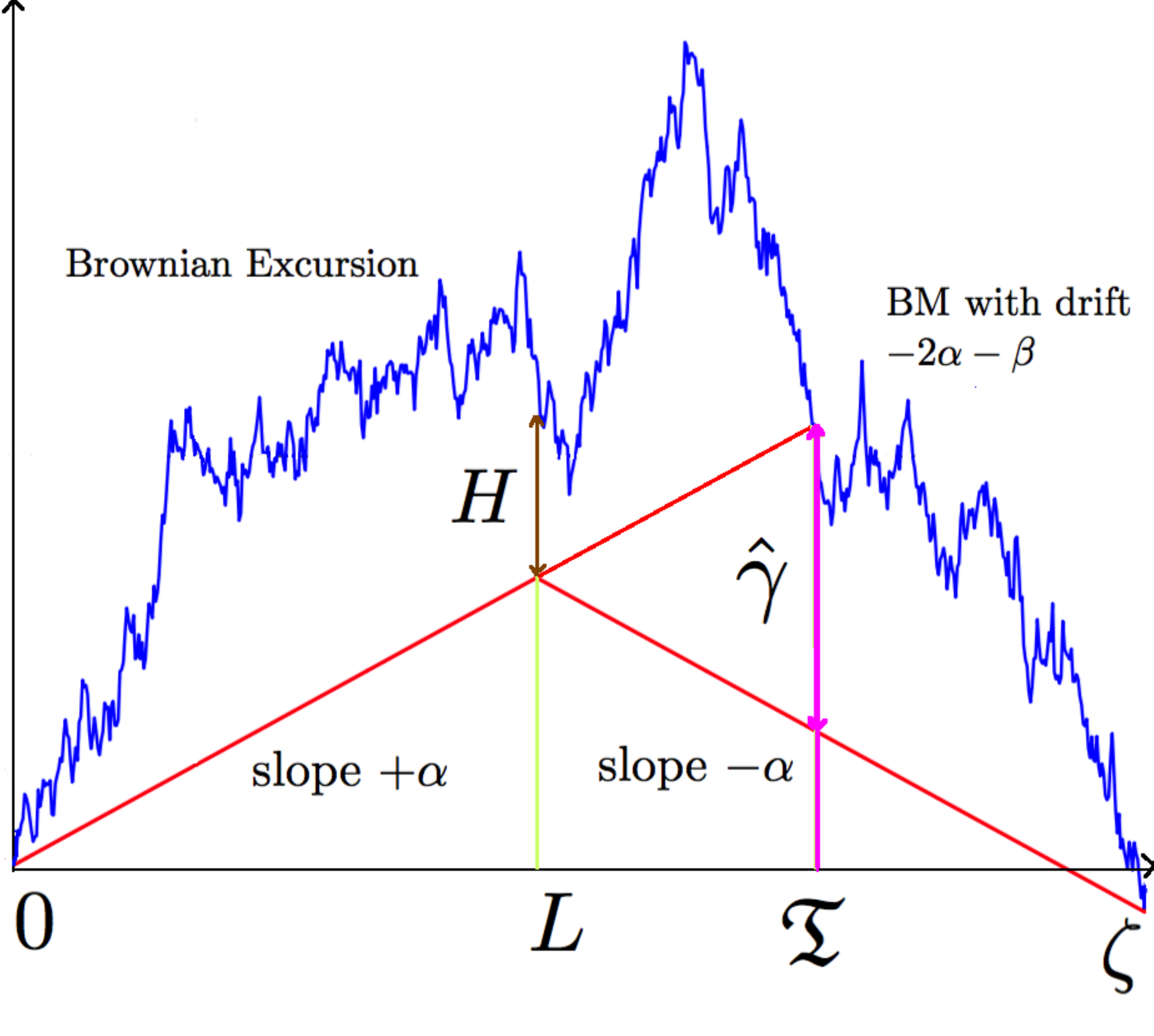}
  \caption{A generic excursion away from the contact set.}
\label{F:excursion}
\end{figure}

Using the notation 
from Theorem ~\ref{lipschitz:decomp} and from Lemma \ref{sawtooth} we have the following expressions
\[
\begin{split}
\zeta& =\tau+\tilde{T}_{\hat{\gamma}}, \\
L&=\tau-\frac{\hat{\gamma}}{2\alpha}, \\
\zeta - L&= \tilde{T}_{\hat{\gamma}} + \frac{\hat{\gamma}}{2\alpha}, \\
W_{\zeta} &=\alpha\left(\tau-\tilde{T}_{\hat{\gamma}}-\frac{\hat{\gamma}}{\alpha}\right). \\
\end{split}
\]

\begin{prop}
\label{laplace}
\begin{itemize}
\item[(i)]
The joint Laplace transform of $(\zeta, L, \zeta - L, W_\zeta)$ is
\[
\begin{split}
& \mathbb{E}[\exp(-(\rho_1 \zeta + \rho_2 L + \rho_3(\zeta - L) + \rho_4 W_\zeta))] \\
& \quad = \frac{4 \alpha}
{2 \alpha 
+ \sqrt{2(\rho_1 + \rho_3 - \alpha \rho_4) + (\alpha + \beta)^2}
+ \sqrt{2(\rho_1 + \rho_2 + \alpha \rho_4) + (\alpha - \beta)^2}}.
\end{split}
\]
\item[(ii)]
The Laplace transform of the excursion length $\zeta$ is
\[
\mathbb{E}[\exp( -\lambda \zeta)]=\frac{4\alpha}{2\alpha+\sqrt{2\lambda +(\alpha+\beta)^2}+\sqrt{2\lambda +(\alpha-\beta)^2}}.
\]
In particular, for $\beta=0$ the probability density of $\zeta$ is
\[
l \mapsto 2\alpha \frac{e^{-\frac{\alpha^2 l}{2}}}{\sqrt{2\pi l}}-2\alpha^2 \overline{\Phi}(\alpha \sqrt{l})
\]
where $\overline{\Phi}(x):=\int_{x}^{\infty} \frac{e^{-u^2/2}}{\sqrt{2\pi}} \, du$.
\item[(iii)]
The Laplace transform of the time $L$ to the peak of the minorant during the excursion is
\[
\E[\exp(-\lambda L)]=\frac{4\alpha}{3\alpha+\beta+\sqrt{2\lambda+(\alpha-\beta)^2}}
\]
The corresponding density is
\[
l \mapsto 4\alpha\frac{e^{-\frac{(\alpha-\beta)^2 l}{2}}}{\sqrt{2\pi l}}-4\alpha(3\alpha+\beta)e^{4\alpha(\alpha+\beta)l}\overline{\Phi}(\sqrt{l}(3\alpha+\beta).
\]
\item[(iv)]
The Laplace transform of the time $\zeta - L$ after the peak of the minorant during the excursion is
\[
\E[\exp(-\lambda (\zeta-L))]=\frac{4\alpha}{3\alpha-\beta+\sqrt{2\lambda+(\alpha+\beta)^2}}
\]
The corresponding density is
\[
l \mapsto 4\alpha\frac{e^{-\frac{(\alpha+\beta)^2 l}{2}}}{\sqrt{2\pi l}}-4\alpha(3\alpha-\beta)e^{4\alpha(\alpha-\beta)l}\overline{\Phi}(\sqrt{l}(3\alpha-\beta).
\]
\item[(v)] The Laplace transform of $W_{\zeta}$, the final value of the excursion, is
\[
\mathbb{E}[\exp( -\lambda W_{\zeta})]=
\frac{4\alpha}
{2 \alpha
+ \sqrt{(\alpha+\beta)^2 - 2\lambda \alpha}
+ \sqrt{(\alpha-\beta)^2 + 2\lambda \alpha}}.
\]
\end{itemize}
\end{prop}

We give the proof of Proposition~\ref{laplace} below after some preparatory results.
We first recall a result about the distribution of the first hitting time of a Brownian motion with drift. \cite[equations 2.0.1 \& 2.0.2, page 295]{handbook}

\begin{lem}
\label{hitting:time}
Let ($B^{(\mu)}_t)_{t \ge 0}$ a Brownian motion with drift $\mu>0$ started at zero. Let $y>0$ and define $T_{\mu,y}:=\inf \{ t \ge 0: B^{(\mu)}_t=y \}$.
The density function of $T_{\mu,y}$ is
\[
f_{T_{\mu,y}}(t)=\frac{y}{\sqrt{2\pi t^3}} \exp\left(-\frac{(y-\mu t)^2}{2t}\right)
\]
and its Laplace transform is
\[
\mathbb{E}[\exp(-\lambda T_{\mu,y})]=e^{-y(\sqrt{2\lambda + \mu^2}-\mu)}.
\]
\end{lem}

For the sake of completeness, we include the proof of the following simple lemma.

\begin{lem}
\label{integral}
For $a,b >0$,
\[
\int_{0}^{\infty} \frac{e^{-at}-e^{-bt}}{\sqrt{2\pi t^3}} \, dt = \sqrt{2b}-\sqrt{2a}.
\]
\end{lem}

\begin{proof}
Suppose without loss of generality that $0 < a < b$.  We have
\[
\begin{split}
 \int_{0}^{\infty} \frac{e^{-at}-e^{-bt}}{\sqrt{2\pi t^3}} \, dt 
& =
\frac{1}{\sqrt{2 \pi}}
\int_0^\infty \int_a^b t e^{-t x} \, dx \, t^{-\frac{3}{2}} \, dt \\
& =
\frac{1}{\sqrt{2 \pi}} \int_a^b \int_0^\infty t^{\frac{1}{2} - 1} e^{-x t} \, dt \, dx \\
& =
\frac{1}{\sqrt{2 \pi}} \int_a^b x^{-\frac{1}{2}} \int_0^\infty u^{\frac{1}{2} - 1} e^{-u} \, du \, dx \\
& = \frac{\sqrt{\pi}}{\sqrt{2 \pi}} \int_a^b x^{-\frac{1}{2}} \, dx \\
& = \sqrt{2b}-\sqrt{2a}, \\
\end{split}
\]
where in the third equality we used the substitution $u = x t$ and in the fourth equality we used the fact that
$\int_0^\infty u^{\frac{1}{2} - 1} e^{-u} \, du = \Gamma(\frac{1}{2}) = \sqrt{\pi}$.
\end{proof}

We now give the proof of Proposition~ \ref{laplace}.

\begin{proof}
We claim that
\begin{equation}
\label{three_variable}
\begin{split}
& \E[\exp (-\lambda_1 \tau -\lambda_2 \hat{\gamma}-\lambda_3 \tilde{T}_{\hat{\gamma}})]\\
& = \frac{4\alpha}{\sqrt{2(\lambda_1-\lambda_3)+4\alpha \lambda_2+(2\alpha+\sqrt{2\lambda_3+(\alpha+\beta)^2})^2} +\sqrt{2\lambda_1+(\alpha-\beta)^2}}. \\
\end{split}
\end{equation}
The stated equation for $\mathbb{E}[\exp(-(\rho_1 \zeta + \rho_2 L + \rho_3(\zeta - L) + \rho_4 W_\zeta))]$ then follows by noting that
\[
\begin{split}
& \rho_1 \zeta + \rho_2 L + \rho_3(\zeta-L) + \rho_4 W_\zeta \\
& \quad =
\rho_1 (\tau+\tilde{T}_{\hat{\gamma}}) + \rho_2 \left(\tau-\frac{\hat{\gamma}}{2\alpha}\right) + \rho_3\left(\tilde{T}_{\hat{\gamma}} + \frac{\hat{\gamma}}{2\alpha}\right) + \rho_4 \alpha\left(\tau-\tilde{T}_{\hat{\gamma}}-\frac{\hat{\gamma}}{\alpha}\right) \\
& \quad =
(\rho_1 + \rho_2 + \alpha \rho_4) \tau
+ \left(-\frac{\rho_2}{2 \alpha} + \frac{\rho_3}{2 \alpha}  - \rho_4\right)\hat{\gamma}
+ (\rho_1 + \rho_3 - \alpha \rho_4) \tilde{T}_{\hat{\gamma}}. \\
\end{split}
\]
The Laplace transforms for the individual random variables follow by specialization and the claimed expressions for densities then follow from standard inversion formulas.

Rather than deriving \eqref{three_variable} we will instead derive directly the Laplace transform of $\zeta$.
This illustrates the method of proof with less notational overhead.
We have
\begin{align*}
& \mathbb{E}[\exp(-\lambda \zeta)] \\
& \quad =\mathbb{E}[e^{-\lambda \tau}\mathbb{E}[e^{-\lambda\tilde{T}_{\hat{\gamma}}}\vert \tau,\hat{\gamma}]]\\
& \quad =\mathbb{E}[e^{-\lambda \tau} e^{-\hat{\gamma}(\sqrt{2\lambda +(\alpha+\beta)^2}-(\alpha+\beta))}]\\
& \quad =\int_{0}^{\infty} \int_{0}^{2\alpha t} e^{-\lambda t} e^{-x(\sqrt{2\lambda +(\alpha+\beta)^2}-(\alpha+\beta))} \frac{\exp\left(-\frac{(\alpha-\beta)^2t}{2}-2(\alpha+\beta)x\right)}{\sqrt{2\pi t^3}}\, dt \, dx \\
& \quad =\frac{1}{\sqrt{2\lambda +(\alpha+\beta)^2}+\alpha+\beta} \int_{0}^{\infty} \frac{e^{-(\lambda +\frac{(\alpha-\beta)^2}{2})t}(1-e^{-2\alpha t(\sqrt{2\lambda +(\alpha+\beta)^2}+\alpha+\beta)})}{\sqrt{2\pi t^3}} \, dt\\
& \quad =\frac{1}{\sqrt{2\lambda +(\alpha+\beta)^2}+\alpha+\beta} \int_{0}^{\infty} \frac{e^{-at}-e^{-bt}}{\sqrt{2\pi t^3}} dt
\end{align*}
for 
\[ a=\lambda +\frac{(\alpha-\beta)^2}{2}, \quad b= \lambda +\frac{(\alpha-\beta)^2}{2}+2\alpha ( \sqrt{2\lambda +(\alpha+\beta)^2}+\alpha+\beta).
\]
A little algebra shows that 
\[
a=\frac{1}{2}(2\lambda +(\alpha-\beta)^2), ~~b=\frac{1}{2}(\sqrt{2\lambda +(\alpha+\beta)^2}+2\alpha)^2.
\]
Hence, using Lemma \ref{integral}, we get that 
\[
\mathbb{E}[\exp(-\lambda \zeta)]=\frac{2\alpha+\sqrt{2\lambda+(\alpha+\beta)^2}-\sqrt{2\lambda+(\alpha-\beta)^2}}{\sqrt{2\lambda +(\alpha+\beta)^2}+\alpha+\beta}.
\]
After multiplying top and bottom by the conjugate this has following simple form
\[
\mathbb{E}[\exp(-\lambda \zeta)]=\frac{4\alpha}{2\alpha+\sqrt{2\lambda +(\alpha+\beta)^2}+\sqrt{2\lambda +(\alpha-\beta)^2}}.
\]
\end{proof}

\begin{remark}
(i) Write $H:=W_L-M_L=\sqrt{\tau} \mathbf{e}(\frac{L}{\tau})$ for the difference between the Brownian motion and its minorant at time $L$ --- see Figure~\ref{F:excursion}.  We can get an explicit description for the distribution of this random variable, though computing either its Laplace transform or density seems tedious to do.  Indeed we know that for every $0 \le u \le 1$, we have that $\mathbf{e}(u) \ed \sqrt{u(1-u)} \chi_3$,
where $\chi_3^2 \ed Q_1^2+Q_2^2+Q_3^2$ for $Q_1,Q_2,Q_3$ three independent standard Gaussian random variables. 
Hence,
\[ H \ed \sqrt{L\left(1-\frac{L}{\tau}\right)} \chi_3=\sqrt{L\left(\frac{\hat{\gamma}}{2\alpha}\right)} \chi_3=\tau\sqrt{\mathfrak{U}\left(1-\mathfrak{U}\right)}\chi_3.
\]
where $\mathfrak{U}:=\frac{\hat{\gamma}}{2\alpha \tau}$. Using the density in Theorem~\ref{lipschitz:decomp} and a change of variable gives that the  joint density of $(\tau,\mathfrak{U})$ at the point $(t,u) \in (0,\infty) \times [0,1]$ is
\[
f_{\tau,\mathfrak{U}}(t,u)=\frac{2\alpha}{\sqrt{2\pi t}} \exp\left(-\frac{(\alpha-\beta)^2}{2}t-\frac{\alpha+\beta}{\alpha} tu\right)
\] and $\chi_3$ independent of $(\tau,\mathfrak{U})$.

(ii) Set $\Psi(\rho_1, \rho_2, \rho_3, \rho_4; \alpha, \beta) = \mathbb{E}[\exp(-(\rho_1 \zeta + \rho_2 L + \rho_3(\zeta - L) + \rho_4 W_\zeta))]$.
From the time-reversal symmetry $(B_t)_{t \in \reals} \ed (B_{-t})_{t \in \reals}$, we expect that
\[
\Psi(\rho_1, \rho_2, \rho_3, \rho_4; \alpha, \beta)
=
\Psi(\rho_1, \rho_3, \rho_2, -\rho_4; \alpha, -\beta),
\]
and this is indeed the case.  This symmetry is somewhat surprising, as it is certainly not apparent from our
path decomposition.  Similarly, from the Brownian scaling 
$(c^{-1} B_{c^2 t})_{t \in \reals} \ed (B_t)_{t \in \reals}$, $c>0$, we expect that
\[
\Psi(\rho_1, \rho_2, \rho_3, \rho_4; \alpha, \beta)
=
\Psi(c^2 \rho_1, c^2 \rho_2, c^2 \rho_3, c^2 \rho_4; c \alpha, c \beta),
\]
and this also holds.

(iii) It follows from the proposition that 
\[
\mathbb{E}[\zeta] = - \frac{d}{d \lambda} \mathbb{E}[\exp(-\lambda \zeta)] |_{\lambda = 0} = \frac{1}{2(\alpha^2 - \beta^2)}.
\]
Similarly, 
\[
\mathbb{E}[L] = \frac{1}{4\alpha(\alpha-\beta)},
\]
\[
\mathbb{E}[\zeta-L] = \frac{1}{4\alpha(\alpha+\beta)},
\]
and
\[
\mathbb{E}[W_\zeta] = \frac{\beta}{2(\alpha^2 - \beta^2)}.
\]
Note that since $\lim_{t \to \infty} (B_t + \beta t)/t = \beta$ almost surely, we expect
$\mathbb{E}[W_\zeta] = \beta \mathbb{E}[\zeta]$ by a renewal--reward argument.

(iv) 
The results of this section advance the study of the excursion straddling zero in the case of the Brownian motion with drift carried out in \cite[Section 8]{zbMATH06288068}. Indeed, the previous study only determined the four-dimensional distribution $(G,D,T,\tilde{H})$, where $T:=\mathrm{argmax}\{ M_t : G \le t \le D \}$ and $\tilde{H}:=X_T-M_T$. Our approach here gives the distribution of the whole path of a generic excursion. Let us define
\[
W^{\mathrm{straddle}} := (X_{t+G}-X_G , \, 0 \le t \le D-G)
 \]
 and 
 \[
W^{\mathrm{generic}} :=(X_{t+D}-X_D, \, 0 \le t \le \zeta).
 \]
By Corollary~\ref{C:size_biasing}, we have
 \[
 \mathbb{E}\left[F(W^{\mathrm{straddle}})\right]=\mathbb{E}[\zeta]^{-1} \mathbb{E}\left[\zeta F(W^{\mathrm{generic}})\right]
 \]
 Because we know the distribution of $W^{\mathrm{generic}}$, the distribution of the straddling excursion can be recovered. In particular, the distribution of $D-G$ is just the size-biasing  of the distribution of $\zeta$; that is,
$E[f(D-G)] = \mathbb{E}[\zeta]^{-1} \mathbb{E}[\zeta f(\zeta)]$ for any nonnegative measurable function $f$. 
For example, the joint Laplace transform of the analogues of $(\zeta, L, \zeta - L, W_\zeta)$ for the straddling excursion is
\[
\frac{-\frac{d}{d\rho_1} \Psi(\rho_1, \rho_2, \rho_3, \rho_4)}{\mathbb{E}[\zeta]}.
\]
Finally, if we denote by $\Lambda$ the L\'evy measure of the subordinator associated with the regenerative set $\mathcal{Z}$, then it has the density given by the following formula
 \[
 \frac{\Lambda(dx)}{\Lambda(\mathbb{R}+)} =\left(2\alpha \frac{e^{-\frac{\alpha^2 x}{2}}}{\sqrt{2\pi x}}-2\alpha^2 \overline{\Phi}(\alpha \sqrt{x})\right) \, dx
 \]
(recall that $\Lambda$ is only defined up to a multiplicative constant).
\end{remark}

\section{Enlargement of the Brownian filtration}
\label{S:enlargement}

In this section, the L\'evy process $(X_t)_{t \in \reals}$ is the standard two-sided linear Brownian motion.  
 Set
\[
\mathcal{\overline{F}}_t:=\sigma \{ X_u : u \le t \} \vee \{ \text{the null sets of } \mathbb{P} \}, \; t \in \reals.
\]
From \cite[Chapter 3, Proposition 2.10]{zbMATH02150787} $(\mathcal{\overline{F}}_t)_{t \in \reals}$ is then right-continuous and $(X_t)_{t \in \reals}$ is a $(\mathcal{\overline{F}}_t)_{t \in \reals}$-two-sided linear standard  Brownian motion. 
We denote $(M_t)_{t \in \reals}$ the $\alpha$-Lipschitz minorant of $X$ and we let $D$ be defined, as above, by
\[
D:= \inf\{ t \geq 0 : X_t=M_t\}.
\]

By Lemma~\ref{L:recipe}, the random time $D$ can be constructed as follows.
  Consider first the stopping time $S$ given by
\[
S=\inf\{ t > 0 : X_t-\alpha t =\inf\{X_u-\alpha u : u \leq 0\} \}.
\]
Then
\[
D=\inf\{ t \geq S : X_t+\alpha t =\inf\{X_u+\alpha u : u \geq S\}\}.
\]
Thus, if we introduce the one-sided Brownian motion $\check X = (X_{t+S}-X_{S})_{t \geq 0}$ which is independent of $\mathcal{\overline{F}}_S$, and we let $\check{T}$ be the time at which the process $(\check{X}_t+\alpha t)_{t\geq 0}$ hits its ultimate infimum (this point is almost surely unique), then
\[
D = S+\check{T}.
\]
As we have seen previously, the random time $D$ is not a stopping time.  However, $D$ is an honest time in the sense of the following definition.

\begin{defn} Let $L$ be a random variable with values in $[0,\infty]$, $L$ is said to be {\em honest} with respect to the filtration $(\mathcal{\overline{F}}_t)_{t \in \reals}$ if, for every $t \ge 0$, there exists an $\mathcal{\overline{F}}_t$-measurable random variable $L_t$ such that on the set $\{L<t\}$ we have $L=L_t$.
\end{defn}

\begin{lem}
The random time $D$ is an honest time.  Moreover, if $T$ is a stopping time, then $\P\{D = T\} = 0$.
\end{lem}

\begin{proof}
We can write $D$ on the event $\{ D < a\}$ as
\[
\begin{split}
D \ind_{D < a}
& =
S\ind_{\{S < a \}} + \inf\{ t \geq 0 : X_{t+S\ind_{\{S<a\}}}-X_{S\ind_{\{S<a\}}}+\alpha t \\
&= \inf\{X_t + \alpha t: S\ind_{\{S<a\}}\leq t \leq a\}\}. \\
\end{split}
\]
The right-hand side is $\mathcal{\overline{F}}_a$-measurable and hence $D$ is an honest time. 
Also, $\P\{D = T\} = 0$ for any stopping time $T$ because $\P\{X_{D+t} > X_D - \alpha t,  \, \forall t > 0\} = 1$ whereas
$\P(\bigcap_{\epsilon > 0} \{ \exists 0 < t < \epsilon, X_{T+t} < X_T - \alpha t\}) = 1$.
\end{proof}
 
 We introduce now a larger filtration that is the smallest filtration containing $(\mathcal{\overline{F}}_t)_{t \in \reals}$ that makes $D$ a stopping time. 

\begin{notation}
For $t \in \reals$, set
 \[
 \mathcal{\overline{F}}_t^{D} :=\bigcap_{\epsilon >0} (\mathcal{\overline{F}}_{t+\epsilon} \vee \sigma(D \wedge (t+\epsilon))).
 \]
\end{notation}

\begin{remark}
For honest times,
\[
\mathcal{\overline{F}}_t^{D}=\{ A \in \mathcal{\overline{F}}_{\infty} : \exists A_t,B_t \in \mathcal{\overline{F}}_t ,  \, A=(A_t \cap \{D>t\})\cup (B_t \cap \{D \leq t \})\}
\]
-- see \cite[Chapter 5]{SPS_1979__13__574_0}.
\end{remark}

Our goal now is to verify that every $(\mathcal{\overline{F}}_t)_{t \ge 0}$-semimartingale remains a $(\mathcal{\overline{F}}_t^D)_{t \ge 0}$-semimartingale, and to give a formula for the canonical semimartingale decomposition in the larger filtration. 

\begin{defn} 
For any random time $\rho$, we call the $(\mathcal{\overline{F}}_t)_{t \ge 0}$-supermartingale defined by
\[ 
Z_t^{\rho}=\mathbb{P}[\rho > t \, \vert \, \mathcal{\overline{F}}_t]
 \]
the {\em Az\'ema supermartingale} associated with $\rho$. We choose versions of the conditional expectations so that this
process is c\`adl\`ag.
\end{defn}

We recall the following result from \cite[Theorem A]{MR509204}.

\begin{thm}
\label{decomp:thm}
Let $L$ be an honest time. A $(\mathcal{\overline{F}}_t)_{t \ge 0}$ local martingale $(\mathfrak{M}_t)_{t \ge 0}$ 
is a semimartingale in the larger filtration $(\mathcal{\overline{F}}_t^L)_{t \ge 0}$ and decomposes as
\[ 
\mathfrak{M}_t=\tilde{\mathfrak{M}}_t + \int_{0}^{t \wedge L} \frac{d \langle \mathfrak{M},Z^L\rangle_s}{Z_{s-}^L} - \int_{L}^t \frac{d\langle \mathfrak{M},Z^L\rangle_s}{1-Z_{s-}^L},
 \]
where $(\tilde{\mathfrak{M}}_t)_{t \geq 0}$ is a $((\mathcal{\overline{F}}_t^L)_{t \ge 0}, \, \mathbb{P})$-local martingale.
\end{thm}

It remains to find an explicit formula for $Z_t^D$.
Define a decreasing sequence of stopping times $(S_n)_{n \geq 0}$ that converges almost surely to $S$ by
\[
S_n := \sum_{k=0}^{\infty} \frac{k+1}{2^n} \ind_{\{\frac{k}{2^n} \leq S < \frac{k+1}{2^n} \}}.
\]
Define the random times $(\check{T}_n)_{n \geq 0}$ by  
\[
\check{T}_n=\sup\{ t \geq 0 : X_{t+S_n}-X_{S_n}+\alpha t = \inf\{ X_{u+S_n}-X_{S_n}+\alpha u, u \geq 0 \}\}.
\]
Note that $\check{T}_n \underset{n \rightarrow \infty}{\rightarrow} \check{T}$ almost surely because 
$\check{T}_n=\mathrm{argmin }\{ X_{u+S}-X_S+\alpha u: u \geq S_n-S \}+S-S_n$ and $\check{T} > 0$ with probability $1$. 
Hence,
\begin{equation*}
\begin{split}
Z_t^D&=\mathbb{P}\{D>t \, \vert  \, \mathcal{\overline{F}}_t\}
=\mathbb{P}\{S+\check{T}>t \, \vert \, \mathcal{\overline{F}}_t\}\\
&=\lim_{n \rightarrow \infty}\mathbb{P}\{\check{T}_n+S_n>t \, \vert \, \mathcal{\overline{F}}_t\}\\
&=\lim_{n \rightarrow \infty} \ind_{\{S_n \geq t\}}+\mathbb{P}\{\check{T}_n>t-S_n, \, S_n \leq t \, \vert \, \mathcal{\overline{F}}_t\}\\
&=\lim_{n \rightarrow \infty}\ind_{\{S_n \geq t \}}+ \sum_{k=0, \frac{k+1}{2^n} \leq t} \mathbb{P}\left\{\check{T}_n>t-\frac{k+1}{2^n}, \, S_n= \frac{k+1}{2^n} \, \vert \, \mathcal{\overline{F}}_t\right\}\\
&=\lim_{n \rightarrow \infty}\ind_{\{S_n \geq t \}}+\sum_{k=0, \frac{k+1}{2^n} \leq t} \mathbb{P}\left\{\check{T}_n>t-\frac{k+1}{2^n} \, \vert \, \mathcal{\overline{F}}_t\right\} \ind_{\{S_n= \frac{k+1}{2^n}\}}\\
\end{split}
\end{equation*}

 If we apply Lemma~\ref{L:countable_stopping} for $\mathfrak{R}:=S_n$ and $\mathfrak{X}:=\ind_{\{\check{T}_n>t-\frac{k+1}{2^n}\}}$  we get
\[
Z_t^D=\lim_{n \rightarrow \infty}  \ind_{\{S_n \geq t \}}
+\sum_{k=0, \frac{k+1}{2^n}\leq t} \mathbb{P}\left\{\check{T}_n>t-\frac{k+1}{2^n}\, \vert \, \mathcal{\check{F}}^{(n)}_{t-\frac{k+1}{2^n}}\right\} \ind_{\{S_n =\frac{k+1}{2^n}\}}
\]
where $(\mathcal{\check{F}}^{(n)}_t)_{t \ge 0} 
=
(\bigcap_{\epsilon>0} \sigma\{X_{u+S_n}-X_{S_n} : 0 \le u \le t+\epsilon\})_{t \ge 0}$.
Now we use the following theorem from \cite[Theorem 8.22]{nikeghbali2006} .

 \begin{prop} 
 \label{azema:decomp}
Let $(N_t)_{t \geq 0}$ be a continuous local martingale such that $N_0=1$ and $\lim_{t \rightarrow \infty} N_t=0$. Let $S_t=\sup_{s \leq t} N_s$. Set 
\[
 g :=\sup\{ t \geq 0 : N_t=S_\infty\}=\sup\{ t \geq 0 : N_t=S_t\}
 \]
Then, the Az\'ema supermartingale associated with the honest time $g$ is given by
 \[
 Z_t^{g}=\mathbb{P}\{g > t \, \vert \, \mathcal{F}_t\}= \frac{N_t}{S_t}.
 \]
 \end{prop}

 We apply Proposition~\ref{azema:decomp} to our case for $g := \check{T}_n$ and the filtration $(\mathcal{\check{F}}^{(n)}_t)_{t \ge 0} 
=
(\bigcap_{\epsilon >0}\sigma\{X_{u+S_n}-X_{S_n} : 0 \le u \le t+\epsilon\})_{t \ge 0}$.

By definition, we have
$\check{T}_n=\sup\{ t \geq 0 : \check{X}^{(n)}_t+\alpha t =\inf\{ \check{X}^{(n)}_u +\alpha u : u \geq 0\}\}$, 
where $\check{X}^{(n)}_t=X_{t+S_n}-X_{S_n}$.  Set
 \[
 N_t=\exp(-2\alpha(\check{X}^{(n)}_t+\alpha t)).
 \]
The process $N$ is clearly a local martingale that verifies the conditions of the last proposition and we also have
 \[
 \check{T}_n=\sup\{ t \geq 0: N_t=\sup_{ s \geq 0} N_s\}.
 \]
 
 Hence
 \[
 \mathbb{P}\{\check{T}_n>t \, \vert \, \mathcal{\check{F}}^{(n)}_t\}= \exp\left(-2\alpha(\check{X}^{(n)}_t+\alpha t)+2\alpha(\inf_{s \leq t} (\check{X}^{(n)}_s+\alpha s))\right).
 \]
 
Finally, we get the expression of the Az\'ema supermartingale associated with $D$ as 
 \[
\begin{split}
 Z_t^D& =\lim_{n \rightarrow \infty} \ind_{\{S_n \geq t \}} \\
& \quad  +\sum_{k=0, \frac{k+1}{2^n}\leq t}^{\infty} \exp\left(-2\alpha\left(\check{X}^{(n)}_{t-\frac{k+1}{2^n}}+\alpha\left (t-\frac{k+1}{2^n}\right)\right)+2\alpha \inf_{ 0 \leq s \leq t-\frac{k+1}{2^n}}\left(\check{X}^{(n)}_s +\alpha s\right)\right)\\
& \qquad \times \ind_{\{S_n =\frac{k+1}{2^n}\}}.
\end{split}
\]
That is,
  \[
\begin{split}
& Z_t^D \\
& \quad =\lim_{n \rightarrow \infty} \ind_{\{S_n \geq t \}} \\
& \qquad +\left[\exp\left(-2\alpha(X_{t}+\alpha (t-S_n)\right)+2\alpha\left(\inf_{s \leq t-S_n} (X_{s+S_n}+\alpha s)\right)\right] \\
& \quad \qquad \times \ind_{\{S_n < t\}}. \\
\end{split}
\]
Thus, by sending $n \rightarrow \infty$, we get that 
 \[
\begin{split}
 & Z_t^D \\
& \quad =\ind_{\{S \geq t \}} \\
& \qquad +\left[\exp\left(-2\alpha(\check{X}_{t-S}+\alpha (t-S)\right)+2\alpha\left(\inf_{s \leq t-S} (\check{X}_{s}+\alpha s)\right)\right]\ind_{\{S < t\}}. \\
\end{split}
\]

Now, using Theorem \ref{decomp:thm}, every $(\mathfrak{M}_t)_{t \ge 0}$ $(\mathcal{\overline{F}}_t)_{t \ge 0}$-local martingale is a $(\mathcal{\overline{F}}_t^D)_{t \ge 0}$-semimartingale and decomposes as follows 
\[ 
\mathfrak{M}_t=\tilde{\mathfrak{M}}_t + \int_{0}^{t \wedge D} \frac{d \langle \mathfrak{M},Z^D\rangle_s}{Z_{s}^D} - \int_{D}^t \frac{d\langle \mathfrak{M},Z^D\rangle_s}{1-Z_{s}^D},
\]
where $(\tilde{\mathfrak{M}}_t)_{t \geq 0}$ denotes a $((\mathcal{\overline{F}}_t^D),\mathbb{P})$-local martingale.

We develop further the expression of $Z^D$ to get an explicit integral representation of its local martingale part. 

\begin{lem}
\label{L:integral_expression}
Let $B$ be a standard Brownian motion and $\alpha >0$. Define the process
$(\mathfrak{H}_t)_{t \ge 0}$ by 
\[
\mathfrak{H}_t=\exp\left(-2\alpha \left[(B_t+\alpha t)-\inf_{s \leq t}(B_s+\alpha s)\right]\right)
\]
Put $I_t=\inf_{s \leq t}(B_s+\alpha s)$.  
Then,
\[
\mathfrak{H}_t=1-2\alpha  \int_{0}^t \mathfrak{H}_u \, dB_u+2 \alpha I_t.
\]
\end{lem}

\begin{proof}
Applying It\^o's formula on the semimartingale $\mathfrak{H}_t=F(B_t+\alpha t, I_t)$, where $F(x,y)=\exp(2\alpha(y-x))$, gives
\begin{equation*}
\begin{split}
 \, d\mathfrak{H}_t&=-2\alpha \mathfrak{H}_tdB_t-2\alpha^2\mathfrak{H}_t dt + 2\alpha \mathfrak{H}_t dI_t + \frac{1}{2} (4\alpha^2) \mathfrak{H}_tdt \\
&=-2\alpha \mathfrak{H}_t dB_t +2\alpha \mathfrak{H}_t dI_t\\
 \, d\mathfrak{H}_t&=-2\alpha \mathfrak{H}_tdB_t+2\alpha dI_t
\end{split}
\end{equation*}
The last line follows from the fact that the measure $dI_t$ is carried on the set $\{t : B_t+\alpha t =I_t\}=\{t : \mathfrak{H}_t=1 \}$.
\end{proof}

Substituting formula from Lemma~\ref{L:integral_expression} into the expression for $Z^D$ we get that 
\[ 
\begin{split}
Z_t^D& =\ind_{\{S \geq t \}} \\
& \quad +\ind_{\{S <t \}}(1-2\alpha  \int_{0}^{t-S}\exp(-2\alpha(\check{X}_{u}+\alpha u)+2\alpha(\inf_{s \leq u}(\check{X}_s+\alpha s)) \, d\check{X}_u  \\
& \quad + 2\alpha \inf_{s \leq t-S} (\check{X}_s+\alpha s)). \\
\end{split}
\]
This can also be written as 
\[ 
\begin{split}
Z_t^D & =1+2 \alpha\ind_{\{S <t \}} \inf_{s \leq t-S} (\check{X}_s+\alpha s) \\
& \quad -2\alpha  \int_{0}^{(t-S)\vee 0}\exp\left(-2\alpha(\check{X}_{u}+\alpha u)+2\alpha\left(\inf_{s \leq u} (\check{X}_s+\alpha s)\right)\right) \, d\check{X}_u. \\
\end{split}
\]

Put $H_u:=\exp(-2\alpha(\check{X}_{u}+\alpha u)+2\alpha(\inf_{s \leq u} (\check{X}_s+\alpha s)))$.
We want to write the integral $ \int_{0}^{(t-S)\vee 0} H_u d\check{X}_u$ as a stochastic integral with respect to the original Brownian motion $X$. For that we consider the time-change $(C_t , \, t \geq 0)$ defined by $C_t :=t+S$. It is clear that this a family of stopping times such that the maps $s \mapsto C_s$ are almost surely increasing and continuous. 
Using  \cite[Chapter V, Proposition 1.5]{zbMATH02150787}, we get that for every bounded $(\mathcal{\overline{F}}_t)_{t \ge 0}$-progressively measurable process $(H_t)_{t \ge 0}$ we have 
\[
\int_{C_0}^{C_t} H_u  \, dX_u= \int_{0}^t H_{C_u}  \, dX_{C_u}.
\]
In our case this becomes
\[
 \int_{S}^{t+S} H_{u-S}  \, dX_u= \int_{0}^t H_{u} \, d\check{X}_u.
\]
Hence,  
\[
\int_{0}^{(t-S)\vee0} H_u \, d\check{X}_u= \int_{S}^{(t-S) \vee 0+S} H_{u-S} \,  \, dX_u= \int_{0}^{t} \ind_{u\geq S}H_{u-S}  \, dX_u.
\]
Finally, 
\[
Z_t^D=1+2 \alpha\ind_{\{S <t \}}\inf_{s \leq t-S} (\check{X}_s+\alpha s)-2\alpha \int_{0}^{t} A_u   \, dX_u,
\]
where  
\[
\begin{split}
A_u & =\ind_{u\geq S}H_{u-S} \\
& =\ind_{u\geq S}\exp\left(-2\alpha(\check{X}_{u-S}+\alpha u)+2\alpha\left(\inf_{s \leq u-S}(\check{X}_s+\alpha s)\right)\right); \\
\end{split}
\]
that is,
\[
A_u=\ind_{u\geq S}\exp\left(-2\alpha(X_{u}+\alpha u)+2\alpha\left(\inf_{S \leq s \leq u} (X_{s}+\alpha s)\right)\right).
\]
The process $t \mapsto 2 \alpha\ind_{\{S <t \}}\inf_{s \leq t-S}(\check{X}_s+\alpha s))$ is decreasing and so the $(\mathcal{\overline{F}}_t)_{t \ge 0}$-local martingale part of $Z^D$ is equal to 
\[
-2\alpha \int_{0}^{t} A_u  \, dX_u.
\]

From the integral representation of martingales with respect to the Brownian filtration (see \cite[Chapter 5, Theorem 3.4]{zbMATH02150787}, every bounded $(\mathcal{\overline{F}}_t)_{t \ge 0}$-martingale $(\mathfrak{M}_t)_{t \ge 0}$ can be written as 
\[
\mathfrak{M}_t=C+ \int_{0}^t \mu_s  \, dX_s.
\]
Such a process decomposes as a $(\mathcal{\overline{F}}_t^D)_{t \ge 0}$-semimartingale in the following way 
\[
\mathfrak{M}_t=\tilde{\mathfrak{M}}_t -2\alpha \int_{0}^{t \wedge D} \frac{\mu_sA_s \, ds}{Z_{s}^D} +2\alpha \int_{D}^t \frac{\mu_sA_s \, ds}{1-Z_{s}^D},
 \]
where $(\tilde{\mathfrak{M}}_t)_{t \ge 0}$ is a $((\mathcal{\overline{F}}_t^D)_{t \ge 0},\mathbb{P})$-local martingale.

\section{General facts about the $\alpha$-Lipschitz minorant}
\label{S:minorants}

Recall that a function $f: \mathbb{R} \mapsto \mathbb{R}$ admits an $\alpha$-Lipschitz minorant  $m$ if and only if $f$ is bounded below on compact sets, $\liminf_{t \rightarrow -\infty} f(t)-\alpha t > -\infty$, and $\liminf_{t \rightarrow +\infty} f(t)+\alpha t > -\infty$.
In this case, 
\begin{equation}
\label{explicit_expression}
m(t)=\inf\{f(s)+\alpha  \vert  t-s  \vert  : s \in \mathbb{R} \}, \quad t \in \reals.
\end{equation}

The following result is obvious from \eqref{explicit_expression}.

\begin{lem}
\label{L:space-time_homogeneous}
Suppose that $f: \reals \to \reals$ is a function with an $\alpha$-Lipschitz minorant.
For $x,s \in \reals$, define $g: \reals \to \reals$ by 
$g = x + f(s + \cdot)$.  
Write $m_f$ and $m_g$ for the respective $\alpha$-Lipschitz minorants of $f$ and $g$.
Then $m_g = x + m_f(s + \cdot)$.
\end{lem}

The next result is a consequence of \cite[Corollary~9.2]{zbMATH06288068} and Lemma~\ref{L:space-time_homogeneous}, but we include a proof for the sake of completeness.

\begin{lem}
Consider a function $f:\reals \to \reals$ for which the $\alpha$-Lipschitz minorant $m$ exists.  Fix $a \in \mathbb{R}$ such that $m(a)=f(a)$.  Define $f^\rightarrow : \reals \to \reals$ by
\[
f^{\rightarrow}(t)
=
\begin{cases}
        f(a)+\alpha(t-a),& t\leq a, \\
        f(t),& t>a. \\
\end{cases}
\]
Denote the $\alpha$-Lipschitz minorant of $f^{\rightarrow}$ by $m^{\rightarrow}$.
Then $m(t)=m^{\rightarrow}(t)$ for all $t\geq a$. 
\end{lem}

\begin{proof}
From the expression of $m^{\rightarrow}$ we have for every $t \geq a$ 
\[
m^{\rightarrow}(t)= \inf\{f(s)+\alpha \vert t-s \vert : s>a\} \wedge (m(a)+\alpha (t-a)).
\]

Note that
\begin{equation*}
\begin{split}
m(a)+\alpha(t-a) &= \inf\{f(s)+\alpha \vert s-a \vert : s \in \mathbb{R} \} +\alpha (t-a)\\
& \leq \inf\{f(s)+\alpha \vert s-a \vert : s  \leq a \} +\alpha (t-a)\\
& = \inf\{f(s)+\alpha (a-s)+\alpha(t-a) : s\leq a \}\\
&= \inf\{f(s) + \alpha \vert t-s \vert : s \leq a \}
\end{split}
\end{equation*}
and so $m^{\rightarrow}(t) \leq m(t)$ for $t \ge a$.

For the reverse inequality, it suffices to prove that 
\[
\inf\{f(s)+\alpha \vert t-s \vert : s \in \mathbb{R} \} \leq m(a)+\alpha(t-a), \quad t \ge a.
\]
By definition, $m(a) \leq f(s) + \alpha \vert s-a \vert $ for all $s \in \mathbb{R}$, and so, by the triangle inequality, 
\[
m(a)+\alpha(t-a) \geq f(s)+\alpha \vert s-a \vert + \alpha \vert a-t \vert \geq f(s) +\alpha \vert t-s \vert
\]
for every $s \in \reals$.
\end{proof}

The following result is \cite[Lemma~9.4]{zbMATH06288068}.

\begin{lem}
\label{L:recipe}
Let $f: \reals \to \reals$ be a c\`adl\`ag function with
$\alpha$-Lipschitz minorant $m : \reals \to \reals$. 
Set 
\[
\mathbf{d} := \inf \{ t>0 : f(t) \wedge f(t-) = m(t) \}, 
\]
\[
\mathbf{s} := 
\inf \left\{  
t > 0 : f(t) \wedge f(t-) -  \alpha t
\leq 
\inf\{ f(u) - \alpha u : u \leq 0 \}  
\right\},
\]
and
\[
\mathbf{e} := 
\inf \left\{ t \ge \mathbf{s} : f(t) \wedge f(t-) + \alpha (t-\mathbf{s}) 
= 
\inf \{ f(u) + \alpha (u-\mathbf{s}) : u \geq \mathbf{s}\}  \right\}.
\]
Suppose that
$f(\mathbf{s}) \le f(\mathbf{s}-)$.
Then, $\mathbf{e}=\mathbf{d}$.
\end{lem}
Let us also state here a simple expression of the time $\mathbf{s}$ when the time zero is a contact point. 
\begin{lem}
\label{L:simplified recip}
Let $f :\reals \to \reals$ be a continuous function with $\alpha$-Lipschitz minorant $m : \reals \to \reals$, and suppose that we have $m(0)=f(0)=0$, then $\mathbf{s}$ defined in 
Lemma~\ref{L:recipe} takes
 the following form
\[
\mathbf{s}=\inf \{ t>0: f(t)=\alpha t\}.
\]
\end{lem}

\begin{proof}
This is straightforward, as 
\begin{equation*}
0 \ge \inf \{ f(u)-\alpha u : u \le 0\} \ge \inf \{ f(u)+\alpha \vert u\vert : u \in \reals\}=m(0)=0
\end{equation*}
because $f(0)=0$.
\end{proof}

The following lemma describes the shape of the $\alpha$-Lipschitz minorant between two consecutive points of the contact set.
It is \cite[Lemma 8.3]{zbMATH06288068}.

\begin{lem}
\label{sawtooth}
Suppose that $f:\reals \to \reals$ that is a c\`adl\`ag with $\alpha$-Lipschitz minorant
$m : \reals \to \reals$. The set $\{t \in \reals : m(t) = f (t)\wedge f(t-) \}$ is closed. If t'<t'' are such that $f(t')\wedge f(t'-)=m(t')$, $f(t'') \wedge f(t''-)=m(t'')$, and $f(t) \wedge f(t-) > m(t)$ for all $t'<t<t''$, then setting $t^{*}=(f(t'')\wedge f(t''-)-f(t')\wedge f(t'-)+\alpha(t''+t'))/(2\alpha)$, 
\[
m(t)
=
\begin{cases}
        f(t')\wedge f(t'-)+\alpha (t-t'),& t' \le t \le t^{*}, \\
        f(t'')\wedge f(t'')+\alpha(t''-t),&, t^{*}\le t \le t''. \\
\end{cases}
\]
\end{lem}

\section{Two random time lemmas}
\label{S:stopping_time}

We detail in this section two lemmas that we used previously in Section~\ref{S:after_first_contact} and  Section~\ref{S:enlargement}. We consider here $(X_t)_{t \in \reals}$ to be two-sided L\'evy process with $(\mathcal{F}_t)_{t \in \reals}$ as its canonical right-continuous filtration; that is, $\mathcal{F}_t := \bigcap_{\epsilon > 0} \sigma\{X_s, \, -\infty < s \le t+\epsilon\}$, $t \in \reals$.

\begin{lem}
\label{L:countable_stopping}
Let $\mathfrak{R}$ 
be a $(\mathcal{F}_t)_{t \in \reals}$-stopping time that takes values in a countable subset of $\mathbb{R}$. 
Define the $\sigma$-fields $\check{\mathcal{F}}_t :=\bigcap_{\epsilon > 0} \sigma( \{X_{u+\mathfrak{R}}-X_{\mathfrak{R}}  : 0 \leq u \leq t+\epsilon\})$, $t \ge 0$, and put $\check{\mathcal{F}}_\infty = \bigvee_{t \ge 0} \check{\mathcal{F}}_\infty$. For every random variable $\mathfrak{X}$ measurable with respect to $\check{\mathcal{F}}_{\infty}$ we have for every $t \in \reals$ and $r \leq t$ that almost surely 
\[
\mathbb{E}\left[\mathfrak{X} \, \vert \, \mathcal{F}_t \right] \ind_{\{\mathfrak{R}=r\}}
=
\mathbb{E}\left[\mathfrak{X}\ind_{\{\mathfrak{R}=r\}} \, \vert \, \mathcal{F}_t \right]
=
\mathbb{E}\left[\mathfrak{X} \, \vert \, \mathcal{\check{F}}_{t-r} \right] \ind_{\{\mathfrak{R}=r\}}.
\]
\end{lem}

\begin{proof}
The first equality is trivial because the event $\{\mathfrak{R}=r\}$ is $\mathcal{F}_r$-measurable and hence $\mathcal{F}_t$-measurable.   We therefore need only prove the second equality.

By a monotone class argument, it suffices to show that the second inequality holds for 
$\mathfrak{X} = \prod_{i=1}^n f_i(X_{u_i+\mathfrak{R}}-X_{\mathfrak{R}})$, where
$0 \leq u_1<u_2< \ldots<u_n$ and $f_1,\ldots,f_n$ are nonnegative Borel functions. 

We have for any $\mathcal{F}_t$-measurable nonnegative random variable $A_t$ that
 \[
\begin{split}
&\mathbb{E}\left[A_t\ind_{\{\mathfrak{R} =r\}} \prod_{i=1}^n(f_i(X_{u_i+\mathfrak{R}}-X_{\mathfrak{R}}))\right] \\
& \quad = \mathbb{E}\left[A_t\ind_{\{\mathfrak{R} =r\}} \prod_{i=1}^n(f_i(X_{u_i+r}-X_r))\right]\\
& \quad =\mathbb{E}\Bigg[A_t\ind_{\{\mathfrak{R} =r\}}\prod_{u_i<t-r} (f_i(X_{u_i+r}-X_r))  \\
 & \qquad \times \mathbb{E}\Bigg[\prod_{u_i\geq t-r} (f_i(X_{u_i+r}-X_r))\, \vert \, \mathcal{F}_t\Bigg]\Bigg]\\
& \quad = 
\mathbb{E}\Bigg[A_t\ind_{\{\mathfrak{R} =r\}}\prod_{u_i<t-r} (f_i(X_{u_i+r}-X_r)) \\
& \qquad \times \mathbb{E}\Bigg[\prod_{u_i\geq t-r}(f_i(X_{u_i+r}-X_t+X_{(t-r)+r}-X_r))\, \vert \, \mathcal{F}_t\Bigg]\Bigg]. \\
\end{split}
\]

Using the independence and stationarity of the increments of the L\'evy process $X$ gives
\[
\begin{split}
& \mathbb{E}\Bigg[A_t\ind_{\{\mathfrak{R} =r\}}\prod_{i=1}^n(f_i(X_{u_i+\mathfrak{R}}-X_{\mathfrak{R}}))\Bigg] \\
& \quad =\mathbb{E}\Bigg[A_t\ind_{\{\mathfrak{R} =r\}}\prod_{u_i<t-r}^n(f_i(X_{u_i+r}-X_r))\\
& \qquad \times \prod_{u_i\geq t-r}^n(g_i(X_{(t-r)+r}-X_r))\Bigg]\\
 \end{split}
\]
for $g_i :=\mathbb{E}\left[f_i(X_{u_i+r-t}+\cdot)\right]$. Thus,
\[
\begin{split}
& \mathbb{E}\Bigg[A_t\ind_{\{\mathfrak{R} =r\}}\prod_{i=1}^n(f_i(X_{u_i+\mathfrak{R}}-X_{\mathfrak{R}}))\Bigg] \\
& \quad =
\mathbb{E}\Bigg[A_t\ind_{\{\mathfrak{R} =r\}}\prod_{u_i<t-r} (f_i(X_{u_i+\mathfrak{R}}-X_{\mathfrak{R}}))\\
&\quad \times  \prod_{u_i\geq t-r} (g_i(X_{(t-r)+\mathfrak{R}}-X_{\mathfrak{R}}))\Bigg].\\
\end{split}
\]

Because the process $(X_{u+\mathfrak{R}}-X_{\mathfrak{R}})_{u \geq 0}$ is itself a L\'evy process with respect to the filtration $(\mathcal{\check{F}}_t)_{t \ge 0}$ and it has the same distribution as $(X_t)_{t \ge 0}$, we have 
\[
\begin{split}
& \mathbb{E}\left[ \prod_{i=1}^n(f_i(X_{u_i+\mathfrak{R}}-X_{\mathfrak{R}})) \, \vert \, \mathcal{\check{F}}_{t-r} \right] \\
& \quad =\prod_{u_i<t-r}^n(f_i(X_{u_i+\mathfrak{R}}-X_{\mathfrak{R}})) \prod_{u_i\geq t-r}^n(g_i(X_{(t-r)+\mathfrak{R}}-X_{\mathfrak{R}})).\\
\end{split}
\]

Thus we finally get the desired equality
\[
\begin{split}
& \mathbb{E}\left[A_t\ind_{\{\mathfrak{R} =r\}} \prod_{i=1}^n(f_i(X_{u_i+\mathfrak{R}}-X_{\mathfrak{R}}))\right] \\
& \quad =\mathbb{E}\left[A_t\ind_{\{\mathfrak{R} =r\}}\mathbb{E}\left[ \prod_{i=1}^n(f_i(X_{u_i+\mathfrak{R}}-X_{\mathfrak{R}})) \, \vert \, \mathcal{\check{F}}_{t-r} \right]\right]. \\
\end{split}
\]
\end{proof}

\begin{lem}
\label{optional:filtration}
Suppose that almost surely $\lim_{t \rightarrow \infty} X_t=\infty$ and that zero is regular for $(0,\infty)$ for the process $(X_t)_{t \in \reals}$. Let $\mathfrak{R}$  be a $(\mathcal{F}_t)_{t \in \reals}$-stopping time. 
Put $(\check X_t)_{t \ge 0} := (X_{t+\mathfrak{R}} - X_{\mathfrak{R}})_{t \ge 0}$.
Consider the random time $\mathfrak{L}:=\sup \{ t \ge 0 : \check{X}_t \wedge \check{X}_{t-}=\inf \{\check{X}_u : u \ge 0\} \}$.
Then, setting  $\mathfrak{D}:=\mathfrak{R}+\mathfrak{L}$, the $\sigma$-field $\sigma\{\check{X}_{t+\mathfrak{L}}-\check{X}_{\mathfrak{L}}: t \ge 0\}$ is independent of the $\sigma$-field
$\mathcal{F}_{\mathfrak{D-}}\vee \sigma \{ X_\mathfrak{D} \}$.
\end{lem}

\begin{proof}
We begin with an observation.
Define the $\sigma$-fields $\check{\mathcal{F}}_t :=\bigcap_{\epsilon > 0} \sigma( \{X_{s+\mathfrak{R}}-X_{\mathfrak{R}}  : 0 \leq s \leq t+\epsilon\})$, $t \ge 0$, and put $\check{\mathcal{F}}_\infty = \bigvee_{t \ge 0} \check{\mathcal{F}}_\infty$. 
It follows from the part of the proof of Theorem~\ref{T:post-D} which comes before we employ the current lemma that  $\sigma\{\check{X}_{t+\mathfrak{L}}-\check{X}_{\mathfrak{L}}: t \ge 0\}$  is independent of
\[
\check{\mathcal{F}}_{\mathfrak{L}}:=\sigma \{ \check{\xi}_{\mathfrak{L}} : \, (\check{\xi}_t)_{t \ge 0} \text{ is an optional process with respect to the filtration } (\check{\mathcal{F}_t})_{t \ge 0} \}.
\]

Returning to the statement of the lemma, and by noticing that $X_D=X_{\mathfrak{R}}+\check{X}_{\mathfrak{L}}$, it suffices to prove for any bounded, nonnegative $\sigma\{\check{X}_{t+\mathfrak{L}}-\check{X}_{\mathfrak{L}}: t \ge 0\}$-measurable random variable $\mathfrak{Y}$, any bounded, nonnegative, continuous functions $g^1,\ldots,g^{n},h^1,h^2$, and any previsible processes $\xi^1,\ldots,\xi^n$ with respect to the filtration $(\mathcal{F}_t)_{t \in \reals}$ that 
\begin{equation}
\label{zeroth:ind}
\mathbb{E}\left[\mathfrak{Y} \prod_{i=1}^n g^{i}(\xi^{i}_{\mathfrak{D}})h^1(X_{\mathfrak{R}})h^2(\check{X}_{\mathfrak{L}})\right] = \mathbb{E}[\mathfrak{Y}] \mathbb{E}\left[\prod_{i=1}^n g^{i}(\xi^{i}_{\mathfrak{D}})h^1(X_{\mathfrak{R}})h^2(\check{X}_{\mathfrak{L}})\right] .
\end{equation}
However, $(\prod_{i=1}^n g^{i}(\xi^{i}_t))_{t \in \reals}$ is itself a previsible process, so it suffices for \eqref{zeroth:ind} to prove for any bounded, nonnegative $\sigma\{\check{X}_{t+\mathfrak{L}}-\check{X}_{\mathfrak{L}}: t \ge 0\}$-measurable random variable $\mathfrak{Y}$, any bounded, nonnegative process $\xi$ that is previsible with respect to filtration $(\mathcal{F}_t)_{t \in \reals}$, and any bounded, nonnegative, continuous functions $h^1,h^2$ that
\begin{equation}
\label{optional:ind}
\mathbb{E}[\mathfrak{Y} \xi_{\mathfrak{D}}h^1(X_{\mathfrak{R}})h^2(\check{X}_{\mathfrak{L}})]=\mathbb{E}[\mathfrak{Y}]\mathbb{E}[\xi_{\mathfrak{D}}h^1(X_{\mathfrak{R}})h^2(\check{X}_{\mathfrak{L}})].
\end{equation}

A stochastic process viewed as a map from $\Omega \times \reals$ to $\mathbb{R}$ is previsible with respect to the filtration $(\mathcal{F}_t)_{t \in \reals}$ if and only if it is measurable with respect to the $\sigma$-field generated by the maps $(\omega,t) \mapsto \ind_{t > T(\omega)}$, where $T$ ranges through the set of $\reals \cup \{+\infty\}$-valued $(\mathcal{F}_t)_{t \in \reals}$-stopping times (see \cite[Chapter IV, Corollary 6.9]{rogerswills} for the analogous fact about previsible processes indexed by $(0,\infty)$). Also, note that the collection of the sets $\mathcal{A}=\{ \{(\omega,t) : t > T(\omega) \}: T \text{ is a stopping time}\}$ is a $\pi$-system because the minimum of two stopping times is a stopping time. Hence,
to establish \eqref{optional:ind}, it suffices by a monotone class argument to show
for any $\reals \cup \{+\infty\}$ -valued $(\mathcal{F}_t)_{t \in \reals}$-stopping time $T$ that
\begin{equation}
\label{second:ind}
\mathbb{E}[\mathfrak{Y}\ind_{\{\mathfrak{D} > T\}}h^1(X_{\mathfrak{R}})h^2(\check{X}_{\mathfrak{L}})]=\mathbb{E}[\mathfrak{Y}]\mathbb{E}[\ind_{\{ \mathfrak{D} > T\}}h^1(X_{\mathfrak{R}})h^2(\check{X}_{\mathfrak{L}})].
\end{equation}
Because we have that $\ind_{\{ D > T\}}=\lim_{n \rightarrow \infty} \ind_{\{D > T \wedge n\}}$, it further suffices to check \eqref{second:ind} for $T$ an $\reals$-valued $(\mathcal{F}_t)_{t \in \reals}$-stopping time.

Consider \eqref{second:ind} in the special case when $\mathfrak{R}$ and $T$ take values in the countable set $\{ r_k:=\frac{k}{2^{n}}, k \in \ints \}$. We then have
\begin{align*}
\mathbb{E}[\mathfrak{Y}\ind_{\{\mathfrak{D} > T\}}h^1(X_{\mathfrak{R}})h^2(\check{X}_{\mathfrak{L}})]&=\mathbb{E}[\mathfrak{Y}\ind_{\{\mathfrak{R}+\mathfrak{L} > T\}}h^1(X_{\mathfrak{R}})h^2(\check{X}_{\mathfrak{L}})]\\
&\quad =\sum_{k,l \in \ints } \mathbb{E}[\mathfrak{Y}\ind_{\{\mathfrak{L} > r_l-r_k\}}\\
& \quad \times h^1(X_{r_k})h^2(\check{X}_{\mathfrak{L}})\ind_{\{\mathfrak{R}=r_k,T=r_l\}}]\\
&=\sum_{ l < k} \mathbb{E}[\mathfrak{Y}\ind_{\{\mathfrak{R}=r_k,T=r_l\}}h^1(X_{r_k})h^2(\check{X}_{\mathfrak{L}})] \\
& \quad +\sum_{ k\le l} \mathbb{E}[\mathfrak{Y}\ind_{\{\mathfrak{L} > r_l-r_k\}}\ind_{\{\mathfrak{R}=r_k,T=r_l\}}\\
& \quad \times h^1(X_{r_k})h^2(\check{X}_{\mathfrak{L}})]\\
&=\sum_{ l < k} \mathbb{E}[\mathfrak{Y}]\mathbb{E}[\ind_{\{\mathfrak{R}=r_k,T=r_l\}} h^1(X_{r_k})h^2(\check{X}_{\mathfrak{L}})] \\
& \quad +\sum_{ k\le l} \mathbb{E}[\mathbb{E}[\mathfrak{Y}\ind_{\{\mathfrak{L} > r_l-r_k\}}h^2(\check{X}_{\mathfrak{L}})\, \vert \, \mathcal{F}_{r_l}]\\
& \quad \times \ind_{\{\mathfrak{R}=r_k,T=r_l\}}h^1(X_{r_k})]
\end{align*}

By applying Lemma ~\ref{L:countable_stopping} for $\mathfrak{X}=\mathfrak{Y}\ind_{\{\mathfrak{L} > r_l-r_k\}}h^2(\check{X}_{\mathfrak{L}})$, we have, for $k \le l$, that
\begin{align*}
& \mathbb{E}[\mathfrak{Y}\ind_{\{\mathfrak{L} > r_l-r_k\}}h^2(\check{X}_{\mathfrak{L}}) \vert \mathcal{F}_{r_l}]\ind_{\{\mathfrak{R}=r_k,T=r_l\}}\\
& \quad =  \mathbb{E}[\mathfrak{Y}\ind_{\{\mathfrak{L} > r_l-r_k\}}h^2(\check{X}_{\mathfrak{L}}) \vert \mathcal{\check{F}}_{r_l-r_k}]\ind_{\{\mathfrak{R}=r_k,T=r_l\}}.
\end{align*}
Moreover, if we let $\check{A}$ to be an event in $\mathcal{\check{F}}_{r_l-r_k}$, then 
\[
\mathbb{E}[\mathfrak{Y} \ind_{\{\mathfrak{L} > r_l-r_k\}\cap \check{A}}h^2(\check{X}_{\mathfrak{L}})]=\mathbb{E}[ \mathfrak{Y}]\mathbb{E}[\ind_{ \{\mathfrak{L} > r_l-r_k \} \cap \check{A}}h^2(\check{X}_{\mathfrak{L}}) ],
\]
because the process $(\check{\xi})_{t \ge 0} =(\ind_{\{ t > r_l-r_k\} \cap \check{A}} h^2(\check{X}_{t}))_{t  \ge 0}$ is clearly an $(\mathcal{\check{F}}_t)_{t \ge 0}$-optional process (as it is the product of the left-continuous, right-limited $(\mathcal{\check{F}}_t)_{t \ge 0}$-adapted process  $(\ind_{ \{t > r_l-r_k\} \cap \check{A}})_{t \ge 0}$ and the c\`adl\`ag $(\mathcal{\check{F}}_t)_{t \ge 0}$-adapted process  $(h^2(\check{X}_t)_{t \ge 0})$).  Hence
\[
\mathbb{E}[\mathfrak{Y} \ind_{\{\mathfrak{L} > r_l-r_k\}} \vert \mathcal{\check{F}}_{r_l-r_k}]=\mathbb{E}[ \mathfrak{Y} ] \mathbb{E}[ \ind_{\{\mathfrak{L} \ge r_l-r_k \}}h^2(\check{X}_{\mathfrak{L}})  \vert \mathcal{\check{F}}_{r_l-r_k}].
\]
Substituting in this equality gives 
\begin{align*}
& \mathbb{E}[\mathfrak{Y} \ind_{\{\mathfrak{D} > T\}}h^1(X_{\mathfrak{R}})h^2(\check{X}_{\mathfrak{L}})] \\
&\quad=\sum_{ l < k} \mathbb{E}[\mathfrak{Y}]\mathbb{E}[\ind_{\{\mathfrak{R}=r_k,T=r_l\}} h^1(X_{\mathfrak{R}})h^2(\check{X}_{\mathfrak{L}})]\\
& \qquad +\sum_{ k\le l} \mathbb{E}[\mathfrak{Y}]\mathbb{E}[\ind_{\{\mathfrak{L} > r_l-r_k, \, R=r_k, \, T=r_l\}}h^1(X_{\mathfrak{R}})h^2(\check{X}_{\mathfrak{L}})]\\
& \quad = \mathbb{E}[\mathfrak{Y}]\mathbb{E}[\ind_{\{\mathfrak{L} > T-\mathfrak{R}\}}h^1(X_{\mathfrak{R}})h^2(\check{X}_{\mathfrak{L}})]\\
& \quad = \mathbb{E}[\mathfrak{Y}]\mathbb{E}[\ind_{\{\mathfrak{D} > T\}}h^1(X_{\mathfrak{R}})h^2(\check{X}_{\mathfrak{L}})].
\end{align*}

We have thus proved \eqref{second:ind} when $\mathfrak{R}$ and $T$ both take values in the set $\{ r_k:=\frac{k}{2^{n}}, k \in \ints \}$.
Suppose now that $T$ is an arbitrary $\reals$-valued stopping time but that $\mathfrak{R}$ still takes values in $\{ r_k:=\frac{k}{2^{n}}, k \in \ints \}$.
For $m \in \nats$ set $T_m := \frac{k}{2^m}$ when $\frac{k-1}{2^m} < T \le \frac{k}{2^m}$, $k \in \ints$.
Thus $(T_m)_{m \in \nats}$ is a decreasing sequence of $(\mathcal{F}_t)_{t \in \reals}$-stopping times converging to $T$.
Taking \eqref{second:ind} with $T$ replaced by $T_m$ and letting $m \to \infty$ we get \eqref{second:ind} for $\mathfrak{R}$ taking values in the set $\{ r_k:=\frac{k}{2^{n}}, k \in \ints \}$ and general $\reals$-valued $T$.

We now to extend to the completely general case of \eqref{second:ind}.
Put $(\check{X}_t^{\mathfrak{R}})_{t \ge 0} :=(X_{t +\mathfrak{R}}-X_{\mathfrak{R}})_{t \ge 0}$. 
Denote the corresponding random variables $\mathfrak{L}$, $\mathfrak{Y}$, and $\mathfrak{D}$  by $\mathfrak{L}^{\mathfrak{R}}$,$\mathfrak{Y}^{\mathfrak{R}}$, and $\mathfrak{D}^{\mathfrak{R}}$, respectively.  Recalling that $\mathfrak{Y}^{\mathfrak{R}}$ is an arbitrary bounded, nonnegative random variable measurable with respect to $\sigma \{ \check{X}_{t+\mathfrak{L}^{\mathfrak{R}}}^{\mathfrak{R}}-\check{X}_{\mathfrak{L}^{\mathfrak{R}}}^{\mathfrak{R}}, t \ge 0\}$, it suffices by a monotone class argument it suffices to show \eqref{second:ind} in the special case where
\[
\mathfrak{Y}^{\mathfrak{R}}=\prod_{i=1}^m f^{i}(\check{X}_{t_i+\mathfrak{L}^{\mathfrak{R}}}^{\mathfrak{R}}-\check{X}_{\mathfrak{L}^{\mathfrak{R}}}^{\mathfrak{R}})=\prod_{i=1}^m f^{i}(X_{t_i+\mathfrak{L}^{\mathfrak{R}}+\mathfrak{R}}-X_{\mathfrak{L}^{\mathfrak{R}}+\mathfrak{R}})
\]
for $f^i$ , $i=1, \dots, m$, bounded, nonnegative, continuous functions and $0 \le t_1 < \dots < t_m$. 

For $n \in \nats$ set $\mathfrak{R}_n := \frac{k}{2^n}$ when $\frac{k-1}{2^n} < \mathfrak{R} \le \frac{k}{2^n}$, $k \in \ints$.
Thus $(\mathfrak{R}_n)_{n \in \nats}$ is a decreasing sequence of $(\mathcal{F}_t)_{t \in \reals}$-stopping times converging to 
$\mathfrak{R}$. 
Note that
\[
\mathfrak{L}^{\mathfrak{R_n}}=\mathrm{argmin} \{ X_{u+\mathfrak{R}}-X_{\mathfrak{R}} : u \ge \mathfrak{R_n}-\mathfrak{R} \}+\mathfrak{R}-\mathfrak{R}_n.
\]
Thus, if $\mathfrak{L}^{\mathfrak{R}}=0$, then $\mathfrak{D}^{\mathfrak{R_n}} \downarrow \mathfrak{D}^{\mathfrak{R}}$ by the right-continuity of the sample paths of $X$. On the other hand,  if $\mathfrak{L}^{\mathfrak{R}}>0$, then, for $n$ large enough, we have that $\mathfrak{D}^{\mathfrak{R_n}}=\mathfrak{D}^\mathfrak{R}$. Hence, by applying the special case of \eqref{second:ind} for the stopping times $\mathfrak{R}_n$  taking discrete values, and using the fact that $X$ has c\`adl\`ag paths we get
\begin{align*}
& \mathbb{E}[\mathfrak{Y}^{\mathfrak{R}} \ind_{\{\mathfrak{D}^{\mathfrak{R}} > T\}}h^1(X_{\mathfrak{R}})h^2(\check{X}_{\mathfrak{L}^{\mathfrak{R}}}^{\mathfrak{R}})]\\
& \quad =  \mathbb{E}\left[\prod_{i=1}^m f^{i}(X_{t_i+\mathfrak{D}^{\mathfrak{R}}}-X_{\mathfrak{D}^{\mathfrak{R}}})
 \ind_{\{\mathfrak{D}^{\mathfrak{R}} > T\}} h^1(X_{\mathfrak{R}})h^2(X_{\mathfrak{D}^{\mathfrak{R}}}-X_{\mathfrak{R}}) \right]\\
& \quad = \lim_{n \rightarrow \infty} \mathbb{E}\left[\prod_{i=1}^m f^{i}(X_{t_i+\mathfrak{D}^{\mathfrak{R_n}}}-X_{\mathfrak{D}^{\mathfrak{R_n}}})
 \ind_{\{\mathfrak{D}^{\mathfrak{R_n}} > T\}}
 h^1(X_{\mathfrak{R_n}})h^2(X_{\mathfrak{D}^{\mathfrak{R_n}}}-X_{\mathfrak{R_n}})\right]\\
& \quad = \lim_{n \rightarrow \infty} \mathbb{E}\left[\prod_{i=1}^m f^{i}(X_{t_i+\mathfrak{D}^{\mathfrak{R_n}}}-X_{\mathfrak{D}^{\mathfrak{R_n}}})
\right] \mathbb{E}[\ind_{\{\mathfrak{D}^{\mathfrak{R_n}} > T\}}
 h^1(X_{\mathfrak{R_n}})h^2(X_{\mathfrak{D}^{\mathfrak{R_n}}}-X_{\mathfrak{R_n}})]\\
& \quad =  \mathbb{E}\left[\prod_{i=1}^m f^{i}(X_{t_i+\mathfrak{D}^{\mathfrak{R}}}-X_{\mathfrak{D}^{\mathfrak{R}}})
\right] \mathbb{E}[\ind_{\{\mathfrak{D}^{\mathfrak{R}} > T\}}
 h^1(X_{\mathfrak{R}})h^2(X_{\mathfrak{D}^{\mathfrak{R}}}-X_{\mathfrak{R}})]\\
& \quad =  \mathbb{E}[\mathfrak{Y}^{\mathfrak{R}}
]\mathbb{E}[\ind_{\{\mathfrak{D}^{\mathfrak{R}} > T\}}  h^1(X_{\mathfrak{R}})h^2(X_{\mathfrak{D}^{\mathfrak{R}}}-X_{\mathfrak{R}})],
\end{align*}
which finishes our proof.
\end{proof}

\def\cprime{$'$}
\providecommand{\bysame}{\leavevmode\hbox to3em{\hrulefill}\thinspace}
\providecommand{\MR}{\relax\ifhmode\unskip\space\fi MR }
\providecommand{\MRhref}[2]{%
  \href{http://www.ams.org/mathscinet-getitem?mr=#1}{#2}
}
\providecommand{\href}[2]{#2}

\end{document}